%% file: NonErgodic.tex
\documentclass[a4,aap]{imsart}

\usepackage{amsfonts}
\usepackage{amssymb}
\usepackage{amsmath,amsthm}
\usepackage[english]{babel}
\usepackage[latin1]{inputenc}
\usepackage{graphicx}
\usepackage{float}
\usepackage{ppr}
\usepackage{mathrsfs,bbm,ifthen,enumerate,color}

\newcommand{\LDsetfunc}{LD-set function }
\newcommand{\xprod}{\bar{x}}
\newcommand{\Xprod}{\bar{X}}
\def\eps{\varepsilon}

\def\rme{\mathrm{e}}

\def\rmd{\mathrm{d}}

\def\PP{\mathbb{P}}

\def\tPP{\PP_\star}
\def\tPE{\PE_\star}

\def\PE{\mathbb{E}}

\def\1{\mathbbm{1}}

\def\CLDset{\mathsf{C}}

\def\Cset{\mathsf{C}}

\def\Kset{\mathsf{K}}
\def\Xset{\mathsf{X}}

\def\Hsigma{\mathcal{H}}

\def\Fsigma{\mathcal{F}}
\def\Xsigma{\mathcal{X}}

\def\JointKernel{T}

\def\osmall{\mathrm{o}}

\def\1{\mathbbm{1}}

\def\ie{\textit{i.e.}}

\newcommand{\wrt}{with respect to }

\newcommand{\iid}{i.i.d.}
\newcommand{\as}{\ensuremath{\mathrm{a.s.}}}

\newcommand{\eg}{e.g.}

\newcommand{\eqsp}{\;}
\newcommand{\lleb}{\lambda^{\mathrm{Leb}}}
\newcommand{\nset}{\mathbb{N}}
\newcommand{\rset}{\mathbb{R}}

\newcommand{\norminfty}[1]{\ensuremath{\left\|#1\right\|_{\infty}}}
\newcommand{\tvnorm}[1]{\ensuremath{\left\|#1\right\|_{\mathrm{TV}}}}

\newcommand{\fprod}{\bar{f}}

\newcommand{\Qprod}{\bar{Q}}

\newcommand{\gprod}{\bar{g}}
\newcommand{\lambdaprod}{\bar{\lambda}}

\newcommand{\CLDsetprod}{\bar{\CLDset}}

\newcommand{\eqdef}{\ensuremath{\stackrel{\mathrm{def}}{=}}}

\def\rme{\mathrm{e}}

\def\PP{\mathbb{P}}

\def\tPP{\PP_\star}
\def\tPE{\PE_\star}

\def\PE{\mathbb{E}}

\def\1{\mathbbm{1}}

\def\CLDset{\mathsf{C}}

\def\Cset{\mathsf{C}}

\def\Kset{\mathsf{K}}
\def\Xset{\mathsf{X}}

\def\Hsigma{\mathcal{H}}

\def\Fsigma{\mathcal{F}}
\def\Xsigma{\mathcal{X}}

\def\JointKernel{T}

\def\osmall{\mathrm{o}}

\def\1{\mathbbm{1}}

\newcommand{\Hsigmaproc}{\ensuremath{\{\Hsigma_k\}_{k\geq 0}}}
\newcommand{\Xinit}{\ensuremath{\nu}}

\newcommand{\Xproc}{\ensuremath{\{X_k\}_{k\geq 0}}}
\newcommand{\epsproc}{\ensuremath{\{\varepsilon_k\}_{k\geq 0}}}
\newcommand{\tepsproc}{\ensuremath{\{\varepsilon_k^*\}_{k\geq 0}}}

\newcommand{\zetaproc}{\ensuremath{\{\zeta_k\}_{k\geq 0}}}
\newcommand{\xiproc}{\ensuremath{\{\xi_k\}_{k\geq 0}}}
\newcommand{\Zproc}{\ensuremath{\{Z_k\}_{k\geq 0}}}

\newcommand{\yseq}{\ensuremath{\{y_i\}_{i=0}^{n}}}
\newcommand{\yvect}{\ensuremath{y_{0:n}}}
\newcommand{\Yset}{\ensuremath{\mathsf{Y}}}
\newcommand{\Ysigma}{\ensuremath{\mathcal{Y}}}
\newcommand{\Yproc}{\ensuremath{\{Y_k\}_{k\geq 0}}}
\newcommand{\tYproc}{\ensuremath{\{Y_k^*\}_{k\geq 0}}}

\newcommand{\XYproc}{\ensuremath{\{X_k,Y_k\}_{k\geq 0}}}

\newcommand{\chunk}[4][]%
{\ifthenelse{\equal{#1}{}}{\ensuremath{{#2}_{#3:#4}}}{\ensuremath{#2^#1}_{#3:#4}}
}

\newcommand{\Q}{\ensuremath{Q}}
\newcommand{\G}{\ensuremath{G}}

\newcommand{\StatDistrib}{\pi}

\newcommand{\vm}{\ensuremath{\varepsilon^-}}
\newcommand{\vp}{\ensuremath{\varepsilon^+}}

\newcommand{\filt}[2][]%
{%
\ifthenelse{\equal{#1}{}}{\ensuremath{\phi_{#2}}}{\ensuremath{\phi_{#1,#2}}}%
}
\newcommand{\pred}[3][]%
{%
\ifthenelse{\equal{#1}{}}{\ensuremath{\phi_{#2|#3}}}{\ensuremath{\phi_{#1,#2|#3}}}%
}
\newcommand{\post}[3][]%
{%
\ifthenelse{\equal{#1}{}}{\ensuremath{\phi_{#2|#3}}}{\ensuremath{\phi_{#1,#2|#3}}}%
}
\newcommand{\logl}[2][]%
{%
\ifthenelse{\equal{#1}{}}{\ensuremath{\ell_{#2}}}{\ensuremath{\ell_{#1,#2}}}%
}
\newcommand{\lhood}[2][]%
{%
\ifthenelse{\equal{#1}{}}{\ensuremath{\mathrm{L}_{#2}}}{\ensuremath{\mathrm{L}_{#1,#2}}}%
}

\newcounter{hyp}

\def\ie{\textit{i.e.}}

\newcommand{\Uproc}{\ensuremath{\{U_k\}_{k\geq 0}}}

\newcommand{\tAproc}{\ensuremath{\{A_k^*\}_{k\geq 0}}}

\begin{document}

\begin{frontmatter}
\title{Forgetting of the initial distribution for non-ergodic Hidden Markov Chains}
\runtitle{Forgetting for non-ergodic HMM}
\runauthor{Gassiat, Landelle, Moulines}
\begin{aug}
\author{Elisabeth Gassiat  \ead[label=a1]{gassiat@math.u-psud.fr}}
\affiliation{Universit\'e Paris-Sud 11, CNRS UMR  8628}
\address{Laboratoire de Mathématiques d'Orsay, \\ Bâtiment 425, 15, rue Georges Clemenceau, \\ 91405 Orsay Cedex Orsay, France. \\\printead{a1}}
\author{Benoit Landelle    \ead[label=a2]{benoit.landelle@math.u-psud.fr}}
\affiliation{Thales Optronique}
\address{Systèmes Aéroportes, \\2 av Gay Lussac, \\78990 ELANCOURT , France \\ \printead{a2}}
\author{Eric Moulines  \corref{}    \ead[label=a3]{eric.moulines@telecom-paristech.fr}}
\affiliation{Institut Télécom / Télécom ParisTech,  UMR CNRS 5181}
\address{Télécom ParisTech, \\46 rue Barrault, 75634, Paris, France.\\ \printead{a3}}
\end{aug}
\begin{abstract}
\paragraph{}
In this paper, the forgetting of the initial distribution for a non-ergodic Hidden Markov Models (HMM) is studied. A new set of conditions is proposed to establish the forgetting property of the filter, which significantly extends all the  existing results. Both a pathwise-type convergence of the total variation distance of the filter started from two different initial distributions, and a convergence in expectation are considered. The results are illustrated using generic models of  non-ergodic HMM and extend all the results known so far.
\end{abstract}
\end{frontmatter}

\input{Introduction}

\input{Results}

\input{Examples}

\input{ThmProofs}

\input{PropSimpleExample}
\input{PropGenericExample}

\end{document}

%% file: Introduction.tex
\section{Introduction and notations}

A Hidden Markov Model (HMM) is a doubly stochastic process with an underlying
Markov chain that is not directly observable. More specifically, let $\Xset$
and $\Yset$ be two spaces equipped with countably generated $\sigma$-fields
$\Xsigma$ and $\Ysigma$; denote by $\Q$ and $\G$ respectively, a Markov
transition kernel on $(\Xset,\Xsigma)$ and a transition kernel from
$(\Xset,\Xsigma)$ to $(\Yset,\Ysigma)$.  Consider the Markov transition kernel
defined for any $(x,y) \in \Xset \times \Yset$ and $C \in \Xsigma \otimes
\Ysigma$ by
\begin{equation}
\label{eq:JointChainHMM}
\JointKernel \left[(x,y), C\right] \eqdef \Q \otimes \G [(x,y),C] = \iint \Q(x,dx') \, \G(x',dy') \1_C(x',y') \eqsp.
\end{equation}
We consider $\XYproc$ the Markov chain with transition kernel $\JointKernel$
and initial distribution $\Xinit \otimes \G(C) \eqdef \iint \Xinit(dx) \G(x,dy)
\1_C(x,y)$, where $\Xinit$ is a probability measure on $(\Xset,\Xsigma)$.  We
assume that the chain $\Xproc$ is not observable (hence the name
\emph{hidden}). In addition, we assume that there exists a
measure $\mu$ on $(\Yset,\Ysigma)$ such that for all $x \in \Xset$,
$\G(x,\cdot)$ is absolutely continuous \wrt\ $\mu$; under these assumptions, the joint
transition kernel $\JointKernel$ may be expressed as
\begin{equation}
\label{eq:JointChain:part_dominatedHMM}
\JointKernel\left[(x,y), C\right] = \iint \Q(x,dx') g(x',y')\, \1_C(x',y')\mu(dy') \eqsp, \quad  C \in \Xsigma \otimes \Ysigma \eqsp,
\end{equation}
where $g(x,\cdot)= \frac{d G(x,\cdot)}{d \mu}$ denotes the Radon-Nikodym
derivative of $G(x,\cdot)$ \wrt\ $\mu$; $g(x,\cdot)$ is referred to as the \emph{likelihood} of the
observation.  We denote by $\filt[\Xinit]{n}[\chunk{y}{0}{n}]$
the distribution of the hidden state $X_n$ conditionally on the observations
$\chunk{y}{0}{n} \eqdef [y_0, \dots, y_n]$, which is given  by
\begin{multline}
  \label{eq:filtering-distribution-1}
  \filt[\Xinit]{n}[\chunk{y}{0}{n}](A) = \frac{\Xinit \left[ g(\cdot,y_0) \Q g(\cdot,y_1) \Q \dots \Q g(\cdot,y_n) \1_A \right]}{\Xinit \left[ g(\cdot,y_0) \Q g(\cdot,y_1) \Q \dots \Q g(\cdot,y_n) \right]} \\
  = \frac{\int_{\Xset^{n+1}} \Xinit(dx_0) g(x_0,y_0) \prod_{i=1}^n
    \Q(x_{i-1},dx_i) g(x_i,y_i) \1_A(x_n)}{\int_{\Xset^{n+1}} \Xinit(dx_0)
    g(x_0,y_0) \prod_{i=1}^n \Q(x_{i-1},dx_i) g(x_i,y_i)} \eqsp,
\end{multline}
where $\Q f(x)= \Q(x,f) \eqdef \int \Q(x,dx') f(x')$, for any function $f \in
\mathbb{B}_+(\Xset)$ the set of non-negative functions $f : \Xset \to \rset$,
such that $f$ is $\Xsigma/\mathcal{B}(\rset)$ measurable, with
$\mathcal{B}(\rset)$ the Borel $\sigma$-algebra.
Let $(\Omega, \Fsigma,\tPP)$ be a probability space and $\Yproc$ be a  $\Yset$-valued stochastic process defined on $(\Omega, \Fsigma)$.

A typical question is under which conditions the distance between the filtering measures $\filt[\Xinit]{n}$ and $\filt[\Xinit']{n}$ for two different choices of the initial distribution $\Xinit$ and $\Xinit'$ vanishes, \ie\
\begin{equation*}
    \lim_{n \to \infty}  \tvnorm{\phi_{\nu,n}[Y_{0:n}]-\phi_{\nu',n}[Y_{0:n}]} =0 \quad \tPP-\as\ \eqsp,
\end{equation*}
where $\tvnorm{\cdot}$ denotes the total variation norm.
We stress that $\Yproc$ is not necessarily itself the observation sequence associated to the HMM used to define the sequence of filtering distribution, which means that we are interested in studying the forgetting property of the initial condition even when the model is mis-specified, which happens to be often the case in practical settings.
The forgetting property of the initial condition of the optimal filter in
nonlinear state space models has attracted many research efforts; see for example the in-depth tutorial of \cite{chigansky:lipster:vanhandel:2008}. The brief overview below is mainly intended to allow comparison of assumptions and results presented in this contributions \wrt\ those previously reported in the literature.

The filtering equation can be seen as a positive random non-linear operator acting on the space of probability measure; the forgetting property can be investigated using tools from the theory of positive operators, namely the Birkhoff contraction inequality for the Hilbert projective metric (see \cite{atar:zeitouni:1997}, \cite{legland:oudjane:2003}, \cite{legland:oudjane:2004}). The results obtained using this approach require stringent \emph{mixing} conditions for the transition
kernels; these conditions state that there exist positive constants
$\epsilon_-$ and $\epsilon_+$ and a probability measure $\lambda$ on
$(\Xset,\Xsigma)$ such that for $f \in \mathbb{B}^+(\Xset)$,
\begin{equation}
\label{eq:mixing-condition}
\epsilon_- \lambda(f) \leq \Q(x,f) \leq \epsilon_+ \lambda(f) \eqsp, \quad \text{for any $x \in \Xset$} \eqsp.
\end{equation}
This condition in particular implies that the chain is uniformly geometrically
ergodic.  Similar results were obtained independently by
\cite{delmoral:guionnet:1998} using the Dobrushin ergodicity coefficient (see
\cite{delmoral:ledoux:miclo:2003} for further refinements under this
assumption). The mixing condition has later been weakened by
\cite{chigansky:lipster:2004}, under the assumption that the kernel $Q$ is
positive recurrent and is dominated by some reference measure $\lambda$:
\[
\sup_{(x,x') \in \Xset \times \Xset} q(x,x') < \infty \quad \text{and} \quad \int \mathrm{ess inf} q(x,x') \StatDistrib(x) \lambda(dx) > 0 \eqsp,
\]
where $q(x,\cdot)= \frac{d \Q(x,\cdot)}{d\lambda}$, $\mathrm{ess inf}$ is the essential infimum \wrt\ $\lambda$ and
$\StatDistrib d \lambda$ is the stationary distribution of the chain $\Q$ . If the upper
bound is reasonable, the lower bound is restrictive in many applications and
fails to be satisfied \eg\ for the linear state space Gaussian model.

In \cite{legland:oudjane:2003}, the stability of the optimal filter is studied
for a class of kernels referred to as \emph{pseudo-mixing}. The definition of
pseudo-mixing kernel is adapted to the case where the state space is $\Xset=
\rset^d$, equipped with the Borel sigma-field $\Xsigma$.  A kernel $\Q$ on
$(\Xset,\Xsigma)$ is \emph{pseudo-mixing} if for any compact set $\Cset$ with a
diameter $d$ large enough, there exist positive constants $\epsilon_-(d) >0$
and $\epsilon_+(d) > 0$ and a  measure $\lambda_\Cset$ (which may
be chosen to be finite without loss of generality) such that
\begin{equation}
\label{eq:pseudo-mixing-kernel}
\epsilon_-(d) \lambda_\Cset(A)\leq \Q(x,A)\leq \epsilon_+(d) \lambda_\Cset(A) \eqsp, \quad \text{for any $x \in \Cset$, $A \in \Xsigma$}
\end{equation}
This condition implies that for any $(x',x'') \in \Cset \times \Cset$,
$$
\frac{\epsilon_-(d)}{\epsilon_+(d)} < \mathrm{essinf}_{x \in \Xset} q(x',x)/q(x'',x)\leq  \mathrm{esssup}_{x \in \Xset} q(x',x)/q(x'',x) \leq \frac{\epsilon_+(d)}{\epsilon_-(d)} \eqsp,
$$
where $q(x,\cdot) \eqdef d \Q(x,\cdot)/ d \lambda_\Cset$, and $\mathrm{esssup}$ and $\mathrm{essinf}$
denote the essential supremum and infimum \wrt\ $\lambda_\Cset$.  This
condition is obviously more general than \eqref{eq:mixing-condition}, but still
it is not satisfied in the linear Gaussian case (see \cite[Example
4.3]{legland:oudjane:2003}).

Several attempts have been made to establish the stability conditions under the
so-called \emph{small} noise condition. The first result in this direction has
been obtained by \cite{atar:zeitouni:1997} (in continuous time) who considered an
ergodic diffusion process with constant diffusion coefficient and linear
observations: when the variance of the observation noise is sufficiently small,
\cite{atar:zeitouni:1997} established that the filter is exponentially stable. Small
noise conditions also appeared (in a discrete time setting) in
\cite{budhiraja:ocone:1999} and \cite{oudjane:rubenthaler:2005}. These results
do not allow to consider the linear gaussian state space model with arbitrary
noise variance.

A very significant step has been achieved by
\cite{kleptsyna:veretennikov:2008}, who considered the filtering problem of
Markov chain $\Xproc$ with values in $\Xset= \rset^d$ filtered from
observations $\Yproc$ in $\Yset= \rset^\ell$,
\begin{equation}
\label{eq:NLGSSM}
\begin{cases}
  X_{k+1} = X_k + b(X_k) + \sigma(X_k) \zeta_k \eqsp, \\
  Y_k = h(X_k) + \beta \varepsilon_k \eqsp.
\end{cases}
\end{equation}
Here $\{ (\zeta_k,\varepsilon_k) \}_{k \geq 0} $ is a \iid\ sequence of standard Gaussian random vectors in $\rset^{d+\ell}$, $b(\cdot)$ is a $d$-dimensional vector
function, $\sigma(\cdot)$ a $d \times d$-matrix function, $h(\cdot)$ is a
$\ell$-dimensional vector-function and $\beta > 0$. The authors established,
under appropriate conditions on $b$, $h$ and $\sigma$, that the optimal filter
forgets the initial conditions; these conditions cover (with some restrictions)
the linear gaussian state space model.

A new approach for ergodic HMM using the so-called \emph{Local Doeblin property} is proposed in \cite{douc:fort:moulines:priouret:2007}. Both almost sure convergence and convergence in expectation for the distance in total variation norm for two filters with different initial distributions  are proven. The results hold under weaker  conditions than those appearing  under other mixing assumptions and, in particular, cover the linear Gaussian state-space model. Moreover, assumptions on observations are relaxed and convergence theorems apply for sequences which are not necessarily HMM.

The works mentioned above mainly deal with ergodic HMM, \ie\ the situations in which the hidden Markov chain is ergodic. Non-ergodic HMM models are routinely used in the non-linear filtering literature, many models used for example in tracking or financial econometrics being simply random walks (see \cite{doucet:defreitas:gordon:2001} and \cite{ristic:arulampalam:gordon:2004} and the references therein).
Non-ergodic HMM have been considered much less frequently in the literature. The main references in this direction are \cite{budhiraja:ocone:1999} and \cite{oudjane:rubenthaler:2005}. In  \cite{budhiraja:ocone:1999}, the observation process is the signal (state) corrupted by an  additive white noise of sufficiently small noise variance. In \cite{oudjane:rubenthaler:2005}, the authors also assumed that the observation  is a possibly non-linear function of the signal (satisfying some additional technical conditions) and that this function of the signal is also observed in an additive noise model of sufficiently small variance. The authors propose to truncate the Markov kernels on random sets depending on the observation sequences, which  are chosen in such a way that the truncated kernels satisfy mixing conditions. The authors establish the convergence of the first-order moment of the difference under signal-to-noise ratio condition.

In this contribution, we propose a new set of conditions to establish the
forgetting property of the filter, which are more general than those proposed
in \cite{budhiraja:ocone:1999} and \cite{oudjane:rubenthaler:2005}.  In Theorem \ref{thm:PathwiseConvergence},  the convergence of
the total variation distance of the filter started from two different initial distributions is established, and is shown to hold almost surely w.r.t.  the probability distribution of the observation process $\Yproc$. Then,
in Theorem \ref{thm:ExpectationConvergence}, a bound for the expectation of this total variation distance is obtained and used in Section \ref{sect:Examples} for nonlinear state-space models to obtain a geometric rate. The results are shown to hold under rather weak conditions on the observation process  $\Yproc$ which do not necessarily entail that the observations are produced by the filtering model.

The paper is organized as followed. In section
\ref{sect:Results}, we introduce the assumptions and state the
main results. In section \ref{sect:Examples}, nonlinear state-space models are considered with different kind of dependence for the state noise and with observations not necessarily produced by the model defining the filter. The proofs are given in sections \ref{sect:ProofsTheorems}, \ref{sect:Proof_SimpleModel}, \ref{sect:Proofs_Generic}.

%% file: Results.tex
\section{Main results}
\label{sect:Results}

In this section, we present two theorems on the forgetting properties of the optimal filter.
These results require the choice of a set-valued function, referred to as \emph{Local Doeblin set function},
which  extends the so-called local Doeblin sets introduced in \cite{veretennikov:2002} and later exploited in \cite{kleptsyna:veretennikov:2008}. The difference between LD-sets of \cite{veretennikov:2002} and  LD-set functions lies in the dependence on the successive observations.
\begin{defn}[\LDsetfunc]
\label{def:LDsetfunc}
A set-valued function $\CLDset: y \longmapsto \CLDset(y)$ from $\Yset$ to $\Xsigma$
is called a \emph{Local Doeblin set function (LD-set function)} if there exists a map $(y,y') \longmapsto
    \big(\varepsilon^-_{\CLDset}(y,y'),\varepsilon^+_{\CLDset}(y,y')\big)$ from $\Yset \times \Yset$ to $(0,\infty)^2$
such that, for all $ (y,y') \in \Yset \times \Yset$, there exists a measure $\lambda_{y,y'}$ on $(\Xset,\Xsigma)$
satisfying
\begin{equation}
\label{eq:LDset-property}
\vm_{\CLDset}(y,y') \lambda_{y,y'}[A\cap \CLDset(y')] \leq
Q[x,A \cap \CLDset(y')] \leq \vp_{\CLDset}(y,y')
\lambda_{y,y'}[A\cap \CLDset(y')]
\end{equation}
for all $x \in \CLDset(y)$ and  $A \in \Xsigma$.
\end{defn}
Some general conditions on the Local Doeblin set function involving the distributions of the observations ensure the forgetting property of the optimal filter.  The case of nonlinear state-space models is studied in Section~\ref{sect:Examples}. Roughly speaking, inequality \eqref{eq:LDset-property} means that the transition of the hidden chain, when the state is in a given subset $C(y)$ does not depend too much on the current state.

We denote, for a set $A \in \Xsigma$ and an observation $y \in \Yset$, the supremum of the likelihood over $A$,
\begin{equation}
\label{eq:Upsilon_A}
\Upsilon_A(y) \eqdef \sup_{x \in A} g(x,y) \eqsp.
\end{equation}
Consider the following assumptions on the likelihood of the observations :
\begin{enumerate}[(H1)]
\item \label{hyp:NotVanishLikelihood_Model}
  For all $(x,y) \in \Xset \times \Yset$, $g(x,y)>0$.
\item \label{hyp:OutsideLDset_Model}
  For all $\eta>0$, there exists an LD-set function $\CLDset_\eta$ satisfying, for all $y \in \Yset$,
    \begin{equation}
        \label{eq:likelihood}
        \Upsilon_{\CLDset_\eta^\comp(y)}(y)\leq \eta \Upsilon_\Xset(y) \eqsp.
    \end{equation}
\end{enumerate}
The first condition states that the likelihood is everywhere positive. This excludes the case of additive noise with bounded support; see for example \cite{budhiraja:ocone:1997}. When $\Xset= \rset^d$, the second assumption is typically satisfied when, for any given $y$, the likelihood goes to zero as the state $|x|$ goes to infinity: $\lim_{|x| \to \infty} g(x,y)= 0$. This assumption is satisfied in many models of practical interest, and roughly implies that the observation effectively provides information on the state range of value.

For a given \LDsetfunc $\CLDset$ , we set
\begin{align}
\label{eq:definition-Phi}
    \Phi_{\nu,\CLDset}(y,y')&\eqdef \nu \left[ g(\cdot,y)Q g(\cdot,y') \mathbf{1}_{\CLDset(y')}(\cdot) \right] \eqsp, \\
\label{eq:definition-Psi}
        \Psi_{\CLDset}(y,y') &\eqdef \lambda_{y,y'} \left[ g(\cdot,y') \mathbf{1}_{\CLDset(y')}\right] \eqsp.
\end{align}

The main idea of the proof is that the states belong very often to the LD-sets. Every time the state is in a LD set and jumps to another LD set, the forgetting mechanism comes into play.
From now on,  for all $(x,x') \in \Xset^2$, denote by $\xprod=(x,x')$ the product $\gprod(\xprod,y)=g(x,y)g(x',y)$. Similarly, for all $A \in \Xsigma$, denote $\bar{A}=  A\times A$, for all LD-set function $\CLDset$,  $\CLDsetprod$ the set-valued function  $\CLDsetprod(y)= \CLDset(y) \times \CLDset(y)$. Finally, for all $(x,x') \in \Xset^2$, and $A$, $B \in \Xsigma$, set $\Qprod(x,x',A \times B)= Q(x,A) Q(x',B)$.
Then, for any $A \in \Xsigma$ and $\nu$ and $\nu'$ two probability measures on $(\Xset,\Xsigma)$, the difference $\phi_{\nu,n}[y_{0:n}](A)-\phi_{\nu',n}[y_{0:n}](A) $ may be expressed as
\begin{align}
\label{eq:difference-filters}
 \phi_{\nu,n}[y_{0:n}](A)&-\phi_{\nu',n}[y_{0:n}](A) \\ \nonumber
 &=
  \frac{\PE_{\nu}^{Q} \left[ \prod_{i=0}^{n}
    g(X_i,y_i) \1_A(X_n) \right]}{\PE_{\nu}^{Q}\left[ \prod_{i=0}^{n}
    g(X_i,y_i) \right]} -  \frac{\PE_{\nu'}^{Q}\left[ \prod_{i=0}^{n}
    g(X_i,y_i) \1_A(X_n) \right]}{\PE_{\nu'}^{Q}\left[ \prod_{i=0}^{n}
    g(X_i,y_i) \right]} \eqsp,  \\ \nonumber
 &= \frac{\PE_{\nu\otimes \nu'}^{\Qprod}\left[ \prod_{i=0}^{n}
    \bar{g}(\bar{X_i},y_i) \1_A(X_n) \right]-\PE_{\nu' \otimes \nu}^{\Qprod}\left[ \prod_{i=0}^{n}
    \bar{g}(\bar{X_i},y_i) \1_A(X_n) \right]}{\PE_{\nu}^{Q}\left[ \prod_{i=0}^{n}
    g(X_i,y_i) \right]\PE_{\nu'}^{Q}\left[ \prod_{i=0}^{n}
    g(X_i,y_i) \right]} \eqsp, \\
 &= \frac{\PE_{\nu\otimes \nu'}^{\Qprod}\left[ \prod_{i=0}^{n}
    \bar{g}(\bar{X_i},y_i) \{ \1_A(X_n) - \1_A(X'_n) \} \right]}{\PE_{\nu}^{Q}\left[ \prod_{i=0}^{n}
    g(X_i,y_i) \right]\PE_{\nu'}^{Q}\left[ \prod_{i=0}^{n}
    g(X_i,y_i) \right]} \eqsp,
\end{align}
We compute bounds for the numerator and the denominator of the previous expression. Such bounds are given in the two following Propositions 
(proofs are postponed to Section \ref{sect:ProofsTheorems}).
For an $LD$-set function $\CLDset$ denotes:
\begin{equation}
\label{eq:definition-rho}
    \rho_{\CLDset}(y,y')\eqdef 1-( \varepsilon^-_{\CLDset} / \varepsilon^+_{\CLDset})^2(y,y') \eqsp.
\end{equation}

\begin{prop}
\label{prop:numerator}
Let $\CLDset$  be an LD-set function and $\nu$ and $\nu'$ be two probability measures on $(\Xset,\Xsigma)$. For any integer $n$ and any sequence $\yseq$ in $\Yset$, let us define
\begin{equation}
\label{eq:Delta_n}
    \Delta_n \big( \nu,\nu',\yvect \big) = \sup_{ A \in \Xsigma } \left| \PE_{\nu\otimes \nu'}^{\Qprod}\left[ \prod_{i=0}^{n}
    \bar{g}(\bar{X_i},y_i) \1_A(X_n) \right]-\PE_{\nu\otimes \nu'}^{\Qprod}\left[ \prod_{i=0}^{n}
    \bar{g}(\bar{X_i},y_i) \1_A(X_n) \right] \right|\eqsp.
\end{equation}
Then,
\begin{equation*}
    \Delta_n \big( \nu,\nu',\yvect \big) \leq \PE_{\nu\otimes \nu'}^{\Qprod} \left\{ \bar{g}(\bar{X_0},y_0)
    \prod_{i=1}^{n}
    \bar{g}(\bar{X}_{i},y_{i}) \rho_{\CLDset(y_{i-1},y_i)}^{\delta_i}  \right\}\eqsp,
\end{equation*}
where $ \delta_i =  \1_{\bar{\CLDset}(y_{i-1}) \times \bar{\CLDset}(y_{i})} (\bar{X}_{i-1},\bar{X}_{i})$.
\end{prop}

\begin{prop}
\label{prop:denominator}
Let  $\CLDset$ be an LD-set function and $\yseq$ a sequence in $\Yset$. We have for all $n \in \nset$
\begin{equation*}
    \PE_{\nu}^{Q} \left[ \prod_{i=0}^{n} g(X_i,y_i)  \right] \geq
    \left( \prod_{i=2}^{n}\varepsilon_{\CLDset}^-(y_{i-1},y_{i}) \right) \Phi_{\nu,\CLDset}(y_0,y_1)
     \prod_{i=2}^{n} \Psi_{\CLDset}(y_{i-1},y_i) \eqsp.
\end{equation*}
\end{prop}

By combining these two Propositions, we obtain an explicit bound for the total variation distance $\tvnorm{\phi_{\nu,n}[y_{0:n}]-\phi_{\nu',n}[y_{0:n}] }$. It is worthwhile
to note that the bound we obtain is valid for any sequence $\chunk{y}{0}{n}$ and any initial distributions $\nu$ and $\nu'$. To state the result, some additional notations are required.
Under assumption (H\ref{hyp:OutsideLDset_Model}), for a fixed $\eta>0$ and a corresponding LD-set function $\CLDset_\eta$, let us define, for $\alpha \in (0,1)$ and a sequence $\yvect = \yseq$ in $\Yset$,
\begin{equation}
\label{eq:defn_Lambda}
    \Lambda_\eta ( \yvect,\alpha) \eqdef \max \left\{ \prod_{k=1}^{n} \rho_{\eta}^{\delta_k}(y_{k-1},y_{k}), \
     \{\delta_k\}_{k=1}^{n} \in \{0,1\}^n \ : \
      \sum_{k=1}^{n} \delta_k \geq \alpha n \right\} \eqsp,
\end{equation}
where $\rho_\eta$ is a shorthand notation for $\rho_{\CLDset_\eta}$ (see \eqref{eq:definition-rho})

\begin{thm}
\label{theo:TVnorm}
Let $\alpha$ be some number in $(0,1)$, $\nu$ and $\nu'$ some probability measures on $(\Xset,\Xsigma)$ and $\yseq$ a sequence in $\Yset$. Then,
\begin{multline}
\label{eq:bound-TVnorm}
    \tvnorm{\phi_{\nu,n}[y_{0:n}]-\phi_{\nu',n}[y_{0:n}] } \leq \Lambda_{\eta}( \yvect,\alpha ) + \\
     \eta^{a_n} \prod_{i=2}^{n}\left( \vm_{\CLDset}(y_{i-1},y_i) \Psi_{\CLDset}(y_{i-1},y_i) \right)^{-2}
    \prod_{i=0}^{n} \Upsilon_\Xset^2(y_i) \Phi_{\nu,\CLDset}^{-1}(y_0,y_1)
    \Phi_{\nu',\CLDset}^{-1}(y_0,y_1)  \eqsp,
\end{multline}
with $a_n=\lfloor \frac{(1-\alpha)n}{2} \rfloor$.
\end{thm}

\begin{proof}
The expression \eqref{eq:difference-filters} together with Proposition \ref{prop:numerator} imply
\begin{equation*}
    \tvnorm{\phi_{\nu,n}[y_{0:n}]-\phi_{\nu',n}[y_{0:n}]} \leq \frac{\Delta_n ( \nu,\nu',\yvect)}{\PE_{\nu}^{Q}\left[ \prod_{i=0}^{n}
    g(X_i,y_i) \right] \PE_{\nu'}^{Q}\left[ \prod_{i=0}^{n}
    g(X_i,y_i) \right]}\eqsp,
\end{equation*}
where $\Delta_n ( \nu,\nu',\yvect )$ is defined by (\ref{eq:Delta_n}). Set
\begin{equation*}
    N_{\CLDsetprod,n}= \sum_{i=1}^{n}  \1\{\Xprod_{i-1} \in \CLDsetprod(y_{i-1})\}  \1\{\Xprod_i \in \CLDsetprod(y_{i})\} \eqsp, \qquad
    M_{\CLDsetprod^\comp,n} = \sum_{i=0}^{n-1} \1_{\CLDsetprod^\comp (y_i)}(\Xprod_i) \eqsp.
\end{equation*}
For any sequence $\{u_j\}$, such that $u_j \in \{0,1\}$ for $j \in \{0,\dots,n\}$ and $u_j=0$ for $j \geq n$,
\begin{equation*}
n \geq \sum_{i=0}^{n-1} u_i \vee u_{i+1} = \sum_{i=0}^{n-1} (u_i +u_{i+1} - u_i  u_{i+1} ) \geq  \sum_{i=0}^n u_i -1 - \sum_{i=0}^{n-1} u_i u_{i+1} \eqsp,
\end{equation*}
which implies that $\sum_{i=0}^{n-1} u_i \leq (n+1)/2 + \sum_{i=0}^{n-1} u_i u_{i+1}$.  Using this inequality with $u_i= \1\{\Xprod_i \in \CLDsetprod(y_{i})\}$ for
$i \in \{0,\dots,n\}$ shows that $N_{\CLDsetprod,n} < \alpha n$ implies that $M_{\CLDsetprod^\comp,n} \geq a_n$. Therefore, using Proposition \ref{prop:numerator}, we obtain
\begin{multline}
\label{eq:Delta_n-decomposed}
    \Delta_n(\nu,\nu',\yvect) \leq  \PE_{\nu\otimes \nu'}^{\Qprod}  \left[  \gprod(\Xprod_0,y_0)  \prod_{i=1}^{n}
    \gprod(\Xprod_{i},y_{i}) \rho^{\delta_i}_\eta(y_{i-1},y_{i}) \1\{N_{\CLDsetprod,n} \geq \alpha n\} \right] \\
     + \PE_{\nu\otimes \nu'}^{\Qprod}\left[ \gprod(\Xprod_0,y_0)  \prod_{i=1}^{n}
    \gprod(\Xprod_{i},y_{i}) \rho^{\delta_i}_\eta(y_{i-1},y_{i}) \1\{N_{\CLDsetprod,n} < \alpha n\}  \right]\eqsp,
\end{multline}
with $ \delta_i =  \1_{\bar{\CLDset}(y_{i-1}) \times \bar{\CLDset}(y_{i})} (\bar{X}_{i-1},\bar{X}_{i})$. The last term in the right-hand side of \eqref{eq:Delta_n-decomposed} satisfies
\begin{multline*}
 \PE_{\nu\otimes \nu'}^{\Qprod} \left[  \gprod(\Xprod_0,y_0)  \prod_{i=1}^{n}
    \gprod(\Xprod_{i},y_{i}) \rho^{\delta_i}_\eta(y_{i-1},y_{i}) \1\{N_{\CLDsetprod,n} < \alpha n\}
    \right]  \\
    \leq  \PE_{\nu\otimes \nu'}^{\Qprod} \left[ \prod_{i=0}^{n}
    \gprod(\Xprod_i,y_i) \1\{M_{\CLDsetprod^\comp,n} \geq a_n\} \right] \eqsp.
\end{multline*}
By splitting this last product, we obtain
\begin{align*}
    &\prod_{i=0}^{n} \gprod(\Xprod_i,y_i) \1\{M_{\CLDsetprod^\comp,n} \geq a_n\} \\
    & \quad   =  \prod_{ \substack{ 0 \leq i \leq n,\\ \Xprod_i \in \CLDsetprod(y_i)^\comp }}  \gprod(\Xprod_i,y_i) \1\{M_{\CLDsetprod^\comp,n} \geq a_n\}
     \times \prod_{ \substack{ 0 \leq i \leq n ,\\ \Xprod_i \in \CLDsetprod(y_i)}}  \gprod(\Xprod_i,y_i) \1\{M_{\CLDsetprod^\comp,n} \geq a_n\} \eqsp,\\
    & \quad  \leq  \eta^{a_n} \times \prod_{ \substack{ 0 \leq i \leq n, \\ \Xprod_i \in \CLDsetprod(y_i)^\comp }} \Upsilon_\Xset^2(y_i) \times
     \prod_{\substack{ 0 \leq i \leq n ,\\ \Xprod_i \in \CLDsetprod(y_i)} } \Upsilon_\Xset^2(y_i) \eqsp,
\end{align*}
which implies $ \PE_{\nu\otimes \nu'}^{\Qprod}\left[ \prod_{i=0}^{n} \gprod(\Xprod_i,y_i) \1\{M_{\CLDsetprod^\comp,n} \geq a_n\} \right] \leq \eta^{a_n} \prod_{i=0}^{n} \Upsilon_\Xset^2(y_i)$.
The first term in the right-hand side expression of \eqref{eq:Delta_n-decomposed} satisfies
\begin{multline*}
    \PE_{\nu\otimes \nu'}^{\Qprod} \left[  \gprod(\Xprod_0,y_0) \prod_{i=1}^{n}
    \gprod(\Xprod_{i},y_{i})  \prod_{i=1}^{n}\rho_\eta^{\delta_i} (y_{i-1},y_{i})\1\{N_{\CLDsetprod,n} \geq \alpha n\} \right] \\
    \leq \PE_{\nu\otimes \nu'}^{\Qprod}\left[ \prod_{i=0}^{n}
    \gprod(\Xprod_i,y_i) \right] \Lambda_{\eta}( \yvect,\alpha)\eqsp.
\end{multline*}
By combining the above relations, the result follows.
\end{proof}

The last step consists in finding conditions upon which the bound in the right hand side of \eqref{eq:bound-TVnorm} is small. This bound depends explicitly on the observations $Y$'s; it is therefore not difficult to state general conditions  upon which this quantity is small.
Let $\Yproc$ be a stochastic process with probability distribution $\tPP$ in $(\Yset,\Ysigma)$, which is not necessarily related to the model under which the filter is computed. We first formulate an almost sure convergence on the total variation distance of the filter initialized with two different probability measures $\nu$ and $\nu'$ and then later establish a convergence of the expectation.
\begin{thm}
\label{thm:PathwiseConvergence}
Assume (H\ref{hyp:NotVanishLikelihood_Model}) and (H\ref{hyp:OutsideLDset_Model}). Assume moreover that there exists some \LDsetfunc $\CLDset$ such that
\begin{align}
\label{hyp:PathwiseConvergence_vm}
&    \liminf_{n \to \infty} n^{-1} \sum_{k=2}^{n} \log \vm_{\CLDset}(Y_{k-1},Y_{k}) >  -M,  && \tPP-\as\\\
\label{hyp:PathwiseConvergence_Upsilon}
&\limsup_{n \to \infty} n^{-1} \sum_{k=0}^{n} \log \Upsilon_\Xset(Y_k)<M, \qquad  && \tPP-\as\  \\
\label{hyp:PathwiseConvergence_Psi}
&\liminf_{n \to \infty} n^{-1} \sum_{k=2}^{n} \log \Psi_{\CLDset} (Y_{k-1},Y_k) > -M, && \tPP-\as\
\end{align}
for some constant $M>0$. Assume in addition that, for all $\eta >0$ and $\alpha \in (0,1)$,
\begin{equation}
\label{hyp:PathwiseConvergence_Lambda}
    \limsup_{n \to \infty} n^{-1} \log \Lambda_\eta ( Y_{0:n},\alpha) <0, \qquad \tPP-\as\
\end{equation}
Then, for any initial probability distributions $\nu$ and $\nu'$ on $(\Xset,\Xsigma)$ such that
\begin{equation*}
    \nu Q \1_{\CLDset(Y_1)}>0\eqsp,\quad \tPP-\as\ \qquad \nu' Q \1_{\CLDset(Y_1)}>0\eqsp,\quad \tPP-\as\
\end{equation*}
we have
\begin{equation*}
    \limsup_{n \to \infty}  n^{-1} \log
    \tvnorm{\phi_{\nu,n}[Y_{0:n}]-\phi_{\nu',n}[Y_{0:n}]} < 0, \quad \tPP-\as\
\end{equation*}
\end{thm}
\begin{proof}[Proof of theorem \ref{thm:PathwiseConvergence}]
Under the stated assumptions, there exists a LD-set function $\CLDset$ and some constant $M>0$ such that
\begin{eqnarray*}
    \limsup_{n \to \infty} \exp(-2Mn) \prod_{i=2}^{n} \left( \vm_{\CLDset}(Y_{i-1},Y_{i}) \right)^{-2} \leq 1\eqsp, && \tPP-\as\\\
    \limsup_{n \to \infty} \exp(-2Mn) \prod_{i=0}^{n}
    \Upsilon_\Xset^2(Y_i) \leq 1\eqsp, && \tPP-\as\ \\
    \limsup_{n \to \infty} \exp(-2Mn) \prod_{i=2}^{n}
    \Psi_{\CLDset}^{-2}(Y_{i-1},Y_i) \leq 1 \eqsp, && \tPP-\as\
\end{eqnarray*}
Let $\alpha$ be some number in $(0,1)$. Since $a_n=\frac{(1-\alpha)n}{2} + \osmall(n)$, by choosing $\eta$ small enough, it follows that
\begin{multline*}
    \limsup_{n \to \infty}
    \eta^{a_n}  \prod_{i=2}^{n}\left[\vm_{\CLDset}(Y_{i-1},Y_i)\Psi_{\CLDset}(Y_{i-1},Y_i) \right]^{-2}
    \prod_{i=0}^{n} \Upsilon_\Xset^2(Y_i) \\
        \leq  \limsup_{n\to \infty} \eta^{a_n}  e^{6Mn} \leq \limsup_{n\to \infty} e^{-c n}
\end{multline*}
for some $c>0$. The proof is concluded by using inequality and \eqref{hyp:PathwiseConvergence_Lambda}.
\end{proof}
The assumptions linking the LD-set function and the observations make this theorem quite abstract. With a filtering model defined by specific equations, assumptions can be directly formulated on the model and on the observations. Such situations will be described through examples presented in Section \ref{sect:Examples}.

Compared to \cite[Theorem 1 ]{douc:fort:moulines:priouret:2007} in the ergodic case , the conditions \eqref{hyp:PathwiseConvergence_vm} and \eqref{hyp:PathwiseConvergence_Lambda} are specific to the non-ergodic case, since they involve the functions $\vm_\CLDset$ and $\vp_\CLDset$. In the ergodic case, these functions are constant and assumptions \eqref{hyp:PathwiseConvergence_vm} and \eqref{hyp:PathwiseConvergence_Lambda} are trivially satisfied.
\begin{thm}
\label{thm:ExpectationConvergence}
Assume (H\ref{hyp:NotVanishLikelihood_Model}) and (H\ref{hyp:OutsideLDset_Model}). Let $\CLDset$ be a LD-set function. Then, for any $M_i>0$, $i=0,\ldots,3,\ \delta >0$ and $\alpha \in (0,1)$, there exist constants $\eta>0$ and $\beta \in (0,1)$ such that, for all $n\in \nset$,
\begin{equation}
\label{eq:thm_ExpectationConvergence}
    \tPE \Big[ \tvnorm{\phi_{\nu,n}[Y_{0:n}]-\phi_{\nu',n}[Y_{0:n}]} \Big] \leq
     2\beta^{n}+r_0(\nu,n) +
    r_0(\nu',n) + \sum_{i=1}^{4} r_i(n)
\end{equation}
where the sequences in the right-hand side of \eqref{eq:thm_ExpectationConvergence} are defined by
\begin{eqnarray}
\label{eq:r_0}
  r_0(\nu,n) &\eqdef& \tPP \left( \log \Phi_{\nu,\CLDset}(Y_0,Y_1)\leq -M_0 n \right), \\
  \label{eq:r_1}
  r_1(n) &\eqdef& \tPP \left( \sum_{k=2}^{n} \log
    \vm_{\CLDset}(Y_{k-1},Y_k) \leq  -M_1 n \right),\\
    \label{eq:r_2}
  r_2(n) &\eqdef& \tPP \left( \sum_{k=0}^{n}\log \Upsilon_\Xset (Y_k) \geq M_2 n \right), \\
  \label{eq:r_3}
  r_3(n) &\eqdef& \tPP \left(  \sum_{k=2}^{n}  \log \Psi_{\CLDset}(Y_{k-1},Y_k)\leq -M_3 n \right),\\
  \label{eq:r_4}
  r_4(n) &\eqdef& \tPP \left(  \log \Lambda_\eta ( Y_{0:n},\alpha) \geq -\delta n \right).
\end{eqnarray}
\end{thm}
The proof is along the same lines as above and left to the reader.
This result does not provide directly a rate of convergence. Indeed, only the first term of the right-hand side of equation \eqref{eq:thm_ExpectationConvergence} gives a geometric rate. In Section \ref{sect:Examples}, for given filtering equations, explicit majorations of the other terms will be obtained with geometric rates. Like for the pathwise convergence, the terms $r_1$ and $r_4$ which involve the functions $\vm_\CLDset$ and $\vp_\CLDset$ are specific to the non-ergodic case.

%% file: Examples.tex
\section{Nonlinear state-space models}
\label{sect:Examples}
Let $\Xset=\rset^n$ and $\Yset=\rset^p$ with $p\leq n$,  endowed with the Borel $\sigma$-algebra $\Xsigma$ and $\Ysigma$. We consider the model:
\begin{equation}
\label{eq:SimpleModel}
              \begin{cases} X_{k}=f(X_{k-1})+\zeta_k \eqsp, \\
              Y_k=h(X_k) + \varepsilon_k \eqsp,
              \end{cases}
\end{equation}
where $f$ and $h$ denote some measurable functions. The observation noise $\epsproc$ is a sequence of i.i.d. random variables with positive density $\upsilon$ \wrt the Lebesgue  measure $\lleb$ on $\Yset$.
We consider the following assumptions:
\begin{enumerate}[(E1)]

  \item \label{hyp:PseudoLipschtiz}
  $f$ is $a$-Lipschitz, i.e. $|f(x)-f(y)| \leq a |x-y|$ and $h$ is  uniformly continuous and surjective and, for all $y_1,\ y_2 \in \Yset$ and $x_1,\ x_2 \in \Xset$ in the preimage of $y_1$ and $y_2$, there exist constant $b_0$ and $b$ such that,
    \begin{equation*}
        |x_1-x_2| \leq b_0 +b|y_1-y_2| \eqsp.
    \end{equation*}

  \item \label{hyp:likelihood}
    The density $\upsilon$  is bounded, and  $\lim_{|u| \to \infty} \upsilon(u)= 0$. Moreover, for all compact set $\Kset \subset \Yset$, the quantity $\inf_{y \in \Kset} \upsilon(y)$ is positive.

\end{enumerate}

Notice that $f$ is not necessarily contracting so that the model is possibly non-ergodic. The assumption (E\ref{hyp:PseudoLipschtiz}) has been first considered in \cite{oudjane:rubenthaler:2005}. A function $f$ satisfying (E\ref{hyp:PseudoLipschtiz}) can be viewed as a perturbation of a bijective function whose inverse is $b$-Lipschtiz. The rationale for considering such assumption is the following. For two successive observations $y_1,\ y_2 \in \Yset$, the distance between inverse images of $y_1,\ y_2$ can not be arbitrarily large. Even if $h$ is not bijective, the distance $|y_1-y_2|$ gives information on the distance of two successives preimage states. The assumption (E\ref{hyp:likelihood}) is more classical and is satisfied, for example, by Gaussian densities.

We first consider the simplest situation where the state noise is a sequence of i.i.d. random variables  independent of the observation noise $\epsproc$ and the observations are distributed according to the model.
Then, we study more general dependence structure of the state noise distribution and the case where the observations do not necessarily follow the model.

\subsection{Nonlinear state-space model with i.i.d. state noise}
\label{subsect:SimpleExample}

In this section, we assume that the state noise $\zetaproc$ is a sequence of i.i.d. random variables with positive density $\gamma$ \wrt\ the Lebesgue measure denoted  $\lleb$ and independent of the observation noise $\epsproc$.  Then, for any $A \in \Xsigma$,
\begin{equation}
\label{eq:SimpleModel:definitionkernel}
    Q(x,A) = \int_{A} \gamma[x'-f(x)]\, \lleb(dx') \eqsp.
\end{equation}
For any $\Delta \in (0,\infty)$, let us define the following set-valued function from $\Yset$ to $\Xsigma$ by
\begin{equation}
\label{eq:SimpleModel_LDsetfunction}
    y \longmapsto \CLDset(y,\Delta) \eqdef \{ x \in \Xset : |h(x)-y| \leq \Delta \} \eqsp.
\end{equation}
 For any $y \in \Yset$, $\CLDset(y,\Delta)$ is included in a neighborhood of the preimage of $y$. Indeed, under assumption (E\ref{hyp:PseudoLipschtiz}), for any $z \in \Xset$ in the preimage of $y$, and any $x \in \CLDset(y,\Delta)$,
\begin{equation*}
    |x-z| \leq b_0 + b \Delta \eqsp.
\end{equation*}
Let $(y,y') \in \Yset^2$. By the condition (E\ref{hyp:PseudoLipschtiz}),
$h$ is surjective so the preimage of $y$ and $y'$ by $h$ is non empty. We choose arbitrarily $z$ and $z'$ in these preimages: $y=f(z)$ and $y'=f(z')$. By the triangle inequality and the
condition (E\ref{hyp:PseudoLipschtiz}), it follows that, for all $(x,x') \in \CLDset(y,\Delta) \times \CLDset(y',\Delta)$,
\begin{equation}
\label{eq:SimpleModel_Transition_Control}
|f(x)-x'| \leq |f(x)-f(z)| + |f(z)-z'| + |z'-x'| \leq a(b_0+b\Delta)+D(y,y') +b_0+b\Delta \eqsp,
\end{equation}
where $D$ is defined by
\begin{equation}
\label{eq:SimpleModel_definition_D}
    D(y,y')\eqdef \sup \big\{ |f(z)-z'| \ : (z,z') \in \Xset^2 \ {\rm with} \  h(z)=y,\  h(z')=y' \big\} \eqsp.
\end{equation}
For any $r > 0$, we consider the minimum and the maximum of the state noise density over a ball of radius $r$:
\begin{equation}
\label{eq:qinf_qsupp_SimpleModel}
    \gamma^-(r) \eqdef  \inf_{|s|\leq r} \gamma(s) \eqsp, \qquad \gamma^+(r) \eqdef \sup_{|s|\leq r} \gamma(s) \eqsp,
\end{equation}
It follows from \eqref{eq:SimpleModel:definitionkernel} and \eqref{eq:SimpleModel_Transition_Control} that, for all $A \in \Xsigma$ and $x \in \CLDset(y,\Delta)$,
\begin{equation}
\label{eq:SimpleModel_LDset_Property}
     \vm_\Delta(y,y')\lleb[A \cap \CLDset(y',\Delta)] \leq
     Q[x,A \cap \CLDset(y',\Delta)] \leq \vp_\Delta(y,y')
      \lleb[A \cap \CLDset(y',\Delta)]\eqsp,
\end{equation}
where,
\begin{align*}
  &\vm_\Delta(y,y')\eqdef\gamma^{-}[(a+1)b_0 + (a+1)b  \Delta + D(y,y')] \eqsp,  \\
  &\vp_\Delta(y,y')\eqdef \gamma^{+}[(a+1)b_0 + (a+1) b d \Delta + D(y,y')] \eqsp.
\end{align*}
Since $\gamma$ is a positive density, it follows by \eqref{eq:SimpleModel_LDset_Property} that the application defined by \eqref{eq:SimpleModel_LDsetfunction} is a LD-set function.
By assumption (E\ref{hyp:likelihood}), for all $\eta>0$, we may choose $\Delta$ large enough so that $\sup_{|s| > \Delta } \upsilon(s) \leq \eta \sup_{s \in \Xset} \upsilon(s)$,
which implies that assumption (H\ref{hyp:OutsideLDset_Model})
\begin{equation}
\label{eq:SimpleModel_Upsilon}
    \Upsilon_{\CLDset^\comp(y,\Delta)}(y)\leq \eta \Upsilon_\Xset(y) \eqsp,
\end{equation}
is satisfied. The positiveness of $\upsilon$ implies  assumption (H\ref{hyp:NotVanishLikelihood_Model}).

To check assumptions \eqref{hyp:PathwiseConvergence_vm} and \eqref{hyp:PathwiseConvergence_Lambda}, it is required to compute an upper bound for $\{D(Y_{k-1},Y_k)\}_{k\geq 1}$. For $z,\ z' \in \Xset$ such that $h(z)=Y_{k-1}, \ h(z')=Y_k$, it follows from the triangle inequality and assumption (E\ref{hyp:PseudoLipschtiz}) that
\begin{eqnarray*}
    |f(z)-z'| &\leq& |f(z)-f(X_{k-1})| + |f(X_{k-1})-X_k| + |X_k-z'| \eqsp,\\
    &\leq&  a (b_0+b|\varepsilon_{k-1}|)+|\zeta_k| + b_0+b|\varepsilon_k|\eqsp.
\end{eqnarray*}
Therefore, for all integer $k \geq 1$,
\begin{equation}
\label{eq:ControlYproc_SimpleExample}
    D(Y_{k-1},Y_k) \leq (a+1)b_0+ab|\varepsilon_{k-1}|+|\zeta_k| +b|\varepsilon_k| \eqsp.
\end{equation}
Thanks to this bound, assumptions \eqref{hyp:PathwiseConvergence_vm} and \eqref{hyp:PathwiseConvergence_Lambda} are satisfied by applying the Law of Large Numbers, see Propositions \ref{prop:SimpleExamplePathwiseCV} and \ref{prop:SimpleExampleExpectationCV} and their proofs. Since $\gamma^-$ is a non increasing function, it follows by \eqref{eq:ControlYproc_SimpleExample} that, for all integer $k\geq 1$, $   \log \vm_\Delta(Y_{k-1},Y_k) \leq  -Z_k^{\Delta}$
where for all $\Delta>0$ and all integer $k \geq 1$,
\begin{equation}
\label{eq:definition-ZkDelta}
    Z_k^{\Delta}\eqdef -\log \gamma^-\left[2(a+1)b_0+(a+1)b\Delta+ab|\varepsilon_{k-1}|+|\zeta_k| +b|\varepsilon_k| \right] \eqsp.
\end{equation}
\begin{prop}
\label{prop:SimpleExamplePathwiseCV}
Let us consider the filtering model defined by \eqref{eq:SimpleModel}.
Assume (E\ref{hyp:PseudoLipschtiz}), (E\ref{hyp:likelihood}) and, for all  $\Delta >0$,
\begin{equation}
\label{hyp:SimpleModel_moment}
\PE|Z_1^{\Delta}|  < \infty \eqsp.
\end{equation}
Let $\Yproc$ be the sequence of observations produced by the filtering equations \eqref{eq:SimpleModel} and let $\CLDset$ be the LD-set function defined by \eqref{eq:SimpleModel_LDsetfunction}.
 Then, for any initial probability distributions $\nu$ and $\nu'$ on $(\Xset,\Xsigma)$  and $\Delta >0$ such that
\begin{equation*}
    \nu Q \1_{\CLDset(Y_1,\Delta)}>0\eqsp, \quad \tPP-\as\ \qquad  \nu' Q \1_{\CLDset(Y_1,\Delta)}>0\eqsp, \quad \tPP-\as\
\end{equation*}
we have
\begin{equation*}
    \limsup_{n \to \infty}  n^{-1} \log
    \tvnorm{\phi_{\nu,n}[Y_{0:n}]-\phi_{\nu',n}[Y_{0:n}]} < 0, \quad \tPP-\as\
\end{equation*}
\end{prop}
The condition \eqref{hyp:SimpleModel_moment}, is not very restrictive.
For example, let us assume that $\gamma$ is a centered Gaussian density and that $\zetaproc$ and $\epsproc$ are sequences of Gaussian random variables.
It follows, that $ \gamma^-(r)=\gamma(r)$ for all $r\geq 0$.
The condition \eqref{hyp:SimpleModel_moment} holds if  $\PE(|\varepsilon_1|^2)<\infty$ and $\PE(|\zeta_1|^2)<\infty$ which are trivially satisfied.

With more stringent conditions for initial laws, geometric rates hold for the convergence of the expected value of the total variation.
Let us recall the definition of the log-moment generating function that will be used in the sequel.
\begin{defn} The log-moment generating function $\psi_Z(\lambda)$  of the random variable $Z$ is defined on the set $\{ \lambda \geq 0 : \PE[\rme^{\lambda Z}]<\infty \}$ by $\psi_Z(\lambda) \eqdef \log \PE[ \rme^{\lambda Z}]$.
\end{defn}

\begin{prop}
\label{prop:SimpleExampleExpectationCV}
Let us consider the filtering model defined by \eqref{eq:SimpleModel} and satisfying (E\ref{hyp:PseudoLipschtiz}), (E\ref{hyp:likelihood}) and,
for all $\Delta >0$, there exists $\tau>0$ such that
\begin{equation}
\label{hyp:SimpleModel_psi}
\text{$\psi_{Z_1^{\Delta}}$ is finite on $[0,\tau)$.}
\end{equation}
Let $\Yproc$ be the sequence produced by the filtering equations \eqref{eq:SimpleModel} and let $\CLDset$ denotes the LD-set function defined by \eqref{eq:SimpleModel_LDsetfunction}. Then, for $\nu$ and $\nu'$ two probability measures on $(\Xset,\Xsigma)$ and $\Delta>0$ such that, for some $\lambda >0$,
\begin{multline}
\label{hyp:InitialLaws_ExpectationConvergence}
    \tPE\left\{ \exp\left(\lambda[\log  \nu g(\cdot,Y_0)Q\1_{\CLDset(Y_1,\Delta)}]_{-}\right) \right\} < \infty \eqsp, \\
    \tPE\left\{ \exp\left(\lambda[\log  \nu' g(\cdot,Y_0)Q\1_{\CLDset(Y_1,\Delta)})]_{-}\right) \right\} < \infty \eqsp,
\end{multline}
we have
\begin{equation*}
    \limsup_{n \to \infty}  n^{-1} \log \tPE \left[ \tvnorm{\phi_{\nu,n}[Y_{0:n}]-\phi_{\nu',n}[Y_{0:n}]} \right] < 0 \eqsp.
\end{equation*}
\end{prop}
Assume that $\gamma$ is the density of a standard Gaussian random variable. The condition $\PE[\rme^{\lambda Z_1^\Delta}]<\infty$ is equivalent to
\begin{equation*}
    \int_{\rset^{n+2p}} \exp \left[(\lambda-\varsigma) |x|^2\right] \, dx < \infty \eqsp,
\end{equation*}
where $\varsigma$ denotes some positive constant. Therefore, for $\lambda>0$ small enough, the condition \eqref{hyp:SimpleModel_psi} is satisfied.

The conditions \eqref{hyp:InitialLaws_ExpectationConvergence} can be interpreted as  non-degenerative conditions. Indeed, they forbid that $\nu g(\cdot,Y_0)Q\1_{\CLDset(Y_1,\Delta)}$ is null almost everywhere and the same for $\nu'$. Intuitively, it means that the distribution of the random variable $\nu g(\cdot,Y_0)Q\1_{\CLDset(Y_1,\Delta)}$ is not concentrated close to zero. For example, if there exists a constant $c>0$ such that
\begin{equation*}
    \nu g(\cdot,Y_0)Q\1_{\CLDset(Y_1,\Delta)}\geq c \eqsp, \quad \tPP-a.s \qquad \nu' g(\cdot,Y_0)Q\1_{\CLDset(Y_1,\Delta)}\geq c\eqsp,  \quad \tPP-a.s
\end{equation*}
then the conditions of \eqref{hyp:InitialLaws_ExpectationConvergence} are satisfied.
Proofs of Propositions \ref{prop:SimpleExamplePathwiseCV} and \ref{prop:SimpleExampleExpectationCV} are given in Section \ref{sect:Proof_SimpleModel}.

\subsection{Nonlinear state-space model with dependent state noise}
\label{subsect:GenericExample}

We now consider the case  where the state noise $\zetaproc$ can depend on previous states. This model has been introduced in \cite[Section 3]{oudjane:rubenthaler:2005} and is important because it covers the case of partially observed discretely sampled diffusion, as well as partially observed stochastic volatility models \cite[Section 2]{budhiraja:ocone:1999}.
This example illustrates that the forgetting property  is kept even when the distributions of the observations differ from the model.

\begin{enumerate}[(G)]

  \item 
  $\zetaproc$ is a sequence of random variables such that, for all integer $k$, $\zeta_k$ is independent of $\varepsilon_k$ and for all $A \in \Xsigma$,
  \begin{equation*}
    \PP(\zeta_k \in A | X_{k-1} =x) = \int q(x,u)  \1_A(u) \, \lleb(du)\eqsp.
  \end{equation*}
  Moreover, there exist a positive probability density $\psi$ and positive constants $\mu^-, \ \mu^+$ such that, for all $x,\ u \in \Xset$,
  \begin{equation*}
        \mu^- \psi(u) \leq q(x,u) \leq \mu^+ \psi(u) \eqsp.
  \end{equation*}
\end{enumerate}
A first example of state equation satisfying (G) is considered in \cite{budhiraja:ocone:1999}. A signal takes its values in $\Xset$ and  follows the equation
\begin{equation}
\label{eq:Budhiraja}
    X_k=f(X_{k-1})+\sigma(X_{k-1})\xi_k \eqsp,
\end{equation}
where $\xiproc$ is a sequence of i.i.d random variables and where $\sigma: \Xset \rightarrow \rset^{n \times n}$ is a measurable function that satisfies, for all $x, \ u \in \Xset$, the following hypoellipticity condition:
\begin{equation}
\label{eq:hypoellipticity}
    \sigma^- |u|^2 \leq \langle u,\sigma(x)\sigma^\T(x) u \rangle \leq \sigma^+ |u|^2 \eqsp,
\end{equation}
where $\sigma^-, \sigma^+$ are positive constants and the superscript $\T$ denotes the transposition.
Another important example where (G) is satisfied is the case of certain discretely sampled diffusions.
Let $(X_t)_{t\geq0}$ be the unique solution of the following stochastic differential equation
\begin{equation*}
    \rmd X_t =\rho(X_t)dt+\sigma(X_t) \rmd B_t \eqsp,
\end{equation*}
where $B$ is the $n$-dimensional Brownian motion and the functions $\rho: \rset^n \to \rset^n$ and $\sigma: \rset^n \to \rset^{n \times n}$ are respectively of class $C^1$  and $C^3$ . Then, the sequence $\Xproc$ satisfies assumption (G) if the function $\sigma$ is hypoelliptic (condition \eqref{eq:hypoellipticity}); see \cite{oudjane:rubenthaler:2005}. The assumptions (E\ref{hyp:PseudoLipschtiz}), (E\ref{hyp:likelihood}) and (G) are a bit more stringent that those made in \cite{oudjane:rubenthaler:2005}. Indeed, in \cite{oudjane:rubenthaler:2005}, the function $h$ is not necessarily uniformly continuous and no restrictions are made on $\upsilon$. This allows to establish the forgetting of the initial condition with probability one without restriction on the signal-to-noise ratio and for sequences of observations which are not necessarily distributed according to the model used to compute the filtering distribution.
Let us denote by $Q$ the transition kernel for $\Xproc$. Then, for all $A \in \Xsigma$ and for all $x\in \Xset$,
\begin{equation*}
    Q(x,A) = \int_{A} q[x,x'-f(x)]\, \lleb(dx') \eqsp.
\end{equation*}
For the same reasons as above,  we consider the same set-valued function $\CLDset$ \eqref{eq:SimpleModel_LDsetfunction} as before.
Let $(y,y') \in \Yset^2$. Like in \eqref{eq:SimpleModel_Transition_Control}, it follows by (E\ref{hyp:PseudoLipschtiz}) and the triangle inequality that, for all $(x,x') \in \CLDset(y,\Delta) \times \CLDset(y',\Delta)$,
\begin{equation*}
    |f(x)-x'| \leq c + d\Delta +D(y,y')\eqsp,
\end{equation*}
where $D$ is defined in \eqref{eq:SimpleModel_definition_D}, $c=(a+1)b_0$ and $d=(a+1)b$. By setting
\begin{equation}
\label{eq:definition_q_pm}
    q^-(r) \eqdef  \mu^- \times \inf_{|v|\leq r} \psi(v) \eqsp, \qquad q^+(r) \eqdef \mu^+ \times \sup_{|v|\leq r} \psi(v) \eqsp,
\end{equation}
it follows from condition (G) that, for all $A \in \Xsigma$ and $x \in \CLDset(y,\Delta)$,
\begin{equation}
\label{eq:definition_varepsilon_pm}
    \varepsilon^-_\Delta(y,y')\lleb[A \cap \CLDset(y',\Delta)] \leq
     Q[x,A \cap \CLDset(y',\Delta)] \leq \varepsilon^+_\Delta(y,y')
      \lleb [A \cap \CLDset(y',\Delta)] \eqsp,
\end{equation}
where
\begin{equation*}
  \varepsilon^-_\Delta(y,y')\eqdef q^{-}[c +d  \Delta + D(y,y')] \eqsp,  \qquad
  \varepsilon^+_\Delta(y,y') \eqdef q^{+}[c + d \Delta + D(y,y')] \eqsp.
\end{equation*}
Since $\psi$ is a positive density, the application defined by \eqref{eq:SimpleModel_LDsetfunction} is a LD-set function. As in Section \ref{subsect:SimpleExample}, assumptions (H\ref{hyp:NotVanishLikelihood_Model}) and (H\ref{hyp:OutsideLDset_Model})  are satisfied.
Assume now that the process $\tYproc$ is generated by the following non-linear state-space observations
\begin{equation}
\label{eq:Obs_HMM}
    \left\{ \begin{array}{l}
              X_k^*=f ^*(X^*_{k-1})+\zeta_k^* \eqsp, \\
              Y^*_k=h^*(X^*_k) + \eps_k^* \eqsp,
            \end{array}
    \right.
\end{equation}
where $\tepsproc$ is a sequence of i.i.d random variables, $f^*$ is $a^*$-Lipschtiz, $h^*$ is surjective and for all $x_1,\ x_2 \in \Xset$,
\begin{equation*}
    |x_1-x_2| \leq b_0^* +b^*|h^*(x_1)-h^*(x_2)|\eqsp,
\end{equation*}
for some positive constants $b_0^*,\ b^*$. For all integer $k\geq 1$ , $\zeta_k^*$ is independent of $\eps_k^*$ and, for all $A \in \Xsigma$,
\begin{equation*}
    \PP( \zeta_k^* \in A | X^*_{k-1} =x) = \int q^*(x,u)  \1_A(u) \, \lleb(du) \eqsp.
\end{equation*}
There exists probability densities $\psi^*$ and positive constants $\mu_-^*, \ \mu_+^*$ such that, for all $x,\ u \in \Xset$,
\begin{equation}
\label{eq:GenExObs_Forgetting}
    \mu_-^* \psi^*(u) \leq q^*(x,u) \leq \mu_*^+ \psi^* (u) \eqsp.
\end{equation}
We assume that
\begin{enumerate}[(O1)]
  \item $f^*$ and $h^*$ are such that  $\norminfty{f -f^*} < \infty$ and $\norminfty{h - h^*}< \infty$.
\end{enumerate}
\begin{lem}
\label{lem:D_Ystar}
Let $\tYproc$ be the sequence following \eqref{eq:Obs_HMM}. Under (O1), for all integer $k \geq 1$,
\begin{equation*}
    D(Y_{k-1}^*,Y^*_{k}) \leq \kappa +2a^*b^* + a^*b^* |\eps_{k-1}^*| +b^*|\eps_k^*|+|\zeta_k^*| \eqsp,
\end{equation*}
where
\begin{equation*}
    \kappa=  \norminfty{f-f^*}+(b_0 + b \norminfty{h^*-h})(1+a^*)
\end{equation*}
\end{lem}
\begin{proof}[Proof of Lemma \ref{lem:D_Ystar}]
For all integer $k \geq 1$, for $z,\ z' \in \Xset$ such that $h(z)=Y_{k-1}^*,\ h(z')=Y_k^*$ and for $u,\ u' \in \Xset$ such that $h^*(u)=Y_{k-1}^*,\ h^*(u')=Y_k^*$, it follows by the triangle inequality that
\begin{eqnarray}
\label{eq:GenEx_TriangleIneq}
    |f(z)-z'| &\leq&  |f(z)-f^*(z)| + |f^*(z)-f^*(u)| + |f^*(u)-u'| + |u'-z'|\eqsp, \notag \\
    &\leq&  \norminfty{f-f^*} + a^*|z-u| + |f^*(u)-u'| + |u'-z'| \eqsp.
\end{eqnarray}
Let us notice that
\begin{equation*}
    |z-u| \leq b_0 + b|h(z)-h(u)| \leq b_0 +b \underbrace{|h(z)-h^*(u)|}_{=0} +b|h^*(u)-h(u)| \eqsp.
\end{equation*}
Then, by denoting $K= b_0 + b \norminfty{h^*-h}$, it follows that $|z-u| \leq K$ and, for the same reasons,  $|z'-u'|\leq K$. Combining these two majorations with \eqref{eq:GenEx_TriangleIneq} leads to
\begin{eqnarray*}
  |f(z)-z'|  & \leq & \kappa+ |f^*(u)-f^*(X_{k-1})| + |f^*(X_{k-1})-X_k| + |X_k-u'| \eqsp, \\
  & \leq & \kappa+a^* [ b_0^*+b^*|h^*(z)-h^*(X_{k-1})| ] + |\zeta_k^*| + b_0^* +b^*|h^*(X_k)-h^*(u')| \eqsp,
\end{eqnarray*}
where $\kappa= \norminfty{f-f^*}+K(1+a^*)$. Thus, it is proven that, for all integer $k\geq 1$,
\begin{equation*}
    D(Y^*_{k-1},Y^*_{k}) \leq K' +2a^*b^* + a^*b^* |\eps_{k-1}^*| +b^*|\eps_{k}^*|+|\zeta_k^*| \eqsp.
\end{equation*}
\end{proof}

Let us define for all $\Delta>0$
\begin{equation}
\label{eq:definition-V*+}
    V^{*\Delta}_+ = \log q^-\left[c+d\Delta +\kappa +2a^*b^* + a^*b^* |\eps_{0}^*| +b^*|\eps_1^*|+|\zeta^*_+| \right]\eqsp,
\end{equation}
where $\zeta^*_+$ is a random variable independent of $\tepsproc$ with density $\psi^*$.
\begin{prop}
\label{prop:GenericExamplePathwiseCV}
Let us consider the filtering model defined by \eqref{eq:SimpleModel} and satisfying (E\ref{hyp:PseudoLipschtiz}), (E\ref{hyp:likelihood}) and (G). Let $\CLDset$ be the LD-set function defined by \eqref{eq:SimpleModel_LDsetfunction} and let $\{ Y_k^* \}_{k \geq 0}$ be the sequence following \eqref{eq:Obs_HMM} such that (O1) holds and, for all $\Delta>0$,
\begin{equation}
\label{eq:Martingale_CV}
    \PE \left( |V^{*\Delta}_+| \log_+ |V^{*\Delta}_+| \right) < \infty \eqsp.
\end{equation}
Then, for any initial probability distributions $\nu$ and $\nu'$ on $(\Xset,\Xsigma)$ and $\Delta>0$ satisfying
\begin{equation*}
    \nu Q \1_{\CLDset(Y^*_1,\Delta)}>0\eqsp, \quad \tPP-\as\ \qquad  \nu' Q \1_{\CLDset(Y^*_1,\Delta)}>0\eqsp, \quad \tPP-\as\
\end{equation*}
we have
\begin{equation*}
    \limsup_{n \to \infty}  n^{-1} \log
    \tvnorm{\phi_{\nu,n}[Y^*_{0:n}]-\phi_{\nu',n}[Y^*_{0:n}]} < 0, \qquad \tPP-\as\
\end{equation*}
\end{prop}
This proposition has important consequences. Observations issued from equations \eqref{eq:SimpleModel} under conditions (E\ref{hyp:PseudoLipschtiz}),  (E\ref{hyp:likelihood}) and (G) are of the  observations produced by \eqref{eq:Obs_HMM} under (O1). It is only needed that  $\norminfty{f -f^*}$ and $\norminfty{h - h^*}$ are bounded to ensure the w.p.1 convergence.

Let us write $\zeta_k^*=g^*(X_{k-1}^*,A_k^*)$ where $g^*$ denotes a measurable function and  $\tAproc$ a sequence of i.i.d. random variables with uniform law on $(0,1)$. We make the following assumptions
\begin{enumerate}

  \item[(O3)]   there exists a measurable function $g^*_+$ such that, for all $x\in \Xset$ and $a\in(0,1)$, $|g^*(x,a)| \leq g^*_+(a)$;

  \item[(O4)] Let $\{Z_k^{*\Delta}\}_{k \geq 0}$ be the sequence defined by, for all $\Delta>0$ and for all integer $k\geq 1$,
      \begin{equation*}
        Z_k^{* \Delta} = -\log q^-\left[c+d\Delta +\kappa +2a^*b^* + a^*b^* |\eps_{k-1}^*| +b^*|\eps_k^*|+g^*_+(U_k^*) \right]\eqsp,
      \end{equation*}
    For all $\Delta>0$, there exists $\tau >0$ such that the log-moment generating function $\Psi_{Z_1^{*\Delta}}$ is finite on $[0,\tau)$.

\end{enumerate}

\begin{prop}
\label{prop:GenericExampleExpectationCV}
Let us consider the filtering model defined by \eqref{eq:SimpleModel} and satisfying (E\ref{hyp:PseudoLipschtiz}), (E\ref{hyp:likelihood}) and (G). Let $\{ Y_k^* \}_{k \geq 0}$ be the sequence following \eqref{eq:Obs_HMM} such that
(O1), (O3) and (O4) hold and let $\CLDset$ be the LD-set function defined by \eqref{eq:SimpleModel_LDsetfunction}. Then, for $\nu$ and $\nu'$ two probability measures on $(\Xset,\Xsigma)$ and $\Delta>0$ such that, for some $\lambda >0$,
\begin{multline*}
    \tPE\left\{ \exp\left(\lambda[\log  \nu g(\cdot,Y_0^*)Q\1_{\CLDset(Y_1^*,\Delta)}]_{-}\right) \right\} < \infty \eqsp, \\
    \tPE\left\{ \exp\left(\lambda[\log  \nu' g(\cdot,Y_0^*)Q\1_{\CLDset(Y_1^*,\Delta)})]_{-}\right) \right\} < \infty \eqsp,
\end{multline*}
we have
\begin{equation*}
    \limsup_{n \to \infty}  n^{-1} \log \tPE \left[ \tvnorm{\phi_{\nu,n}[Y^*_{0:n}]-\phi_{\nu',n}[Y^*_{0:n}]} \right] < 0 \eqsp.
\end{equation*}
\end{prop}

For the convergence in expectation, the restrictive assumption (O3) has to be made. Let us precise that, for the case considered in \cite{budhiraja:ocone:1999}, this condition is satisfied since the function $\sigma$ in \eqref{eq:Budhiraja} is bounded. The case of \cite{oudjane:rubenthaler:2005} is not covered by this condition. It seems quite difficult to get the same results as in \cite{oudjane:rubenthaler:2005} with observations not necessarily from an HMM without strengthening the assumptions on $\{ \zeta_k^*\}_{k \geq 0}$. Let us precise that the convergence theorem of \cite{oudjane:rubenthaler:2005} is proved for observations issued from the filtering equations. The assumption (O4) is of the same type as \eqref{hyp:SimpleModel_psi}.

Proofs of Propositions \ref{prop:GenericExamplePathwiseCV} and \ref{prop:GenericExampleExpectationCV}  are given in Section \ref{sect:Proofs_Generic}.

%% file: ThmProofs.tex
\section{Proofs of Propositions \ref{prop:numerator} and \ref{prop:denominator}}
\label{sect:ProofsTheorems}


\begin{proof}[Proof of Proposition \ref{prop:numerator}] For  convenience, we write $\CLDset_i=\CLDset(y_i)$, $\varepsilon_i^- = \varepsilon_{\CLDset}^-(y_{i-1},y_{i})$, $\varepsilon_i^+= \varepsilon_{\CLDset}^+(y_{i-1},y_{i})$, $g_i(x)=g(x,y_i)$, $\lambda_i=\lambda_{y_{i-1},y_{i}}$ and $\rho_i =1-(\varepsilon_i^- / \varepsilon_i^+)^2$. Let us define $\lambdaprod_i \eqdef \lambda_i \otimes \lambda_i$.
Since $\CLDset$ is an LD-set function, for all $i=1,\ldots,n$, $\bar{x} \in \CLDsetprod_{i-1}$, and  $\bar{f}$ a non-negative function on $\Xset\times\Xset$,
\begin{equation}
\label{eq:LDsetprod-property}
 (\varepsilon_i^{-})^2  \lambdaprod_{i}( \1_{\CLDsetprod_{i}} \bar{f})
    \leq  \bar{Q}(\bar{x},\1_{\CLDsetprod_{i}}\bar{f})  \leq
    (\varepsilon_i^{+})^2  \lambdaprod_{i}( \1_{\CLDsetprod_{i}} \bar{f})\eqsp.
\end{equation}
Let us define the sequence of unnormalized kernels $\bar{Q}_i^0$ and $\bar{Q}_i^1$ by, for all $\bar{x} \in \Xset^2$, and
$\bar{f}$ a non-negative function on $\Xset \times \Xset$,
\begin{align*}
    \Qprod_i^0(\xprod,\fprod) &= (\varepsilon_i^-)^2\1_{\CLDsetprod_{i-1}} \lambdaprod_i ( \1_{\CLDsetprod_{i}} \fprod )\eqsp, \\
    \Qprod_i^1( \xprod,\fprod) &= \Qprod(\xprod,\fprod)- (\varepsilon_i^-)^2\1_{\CLDsetprod_{i-1}} \lambdaprod_i ( \1_{\CLDsetprod_{i}} \fprod )\eqsp.
\end{align*}
It follows from (\ref{eq:LDsetprod-property}) that, for all $\xprod$ in $\CLDsetprod_{i-1}$, $ 0 \leq \Qprod_i^1(\xprod,\1_{\CLDsetprod_{i}} \fprod ) \leq \rho_i  \Qprod(\xprod, \1_{\CLDsetprod_{i}} \fprod)$ which implies that, for all $\xprod \in \Xset^2$,
\begin{eqnarray*}
  \Qprod_i^1( \xprod,\fprod) &=& \1_{\CLDsetprod_{i-1}}(\xprod)\Qprod_i^1(\xprod,\1_{\CLDsetprod_{i}} \fprod ) +
  \1_{\CLDsetprod_{i-1}}(\xprod)\Qprod_i^1(\xprod,\1_{\CLDsetprod_{i}^\comp} \fprod ) + \1_{\CLDsetprod^\comp_{i-1}}(\xprod)\Qprod_i^1(\xprod,\fprod ) \eqsp,\\
   &\leq&  \rho_i \1_{\CLDsetprod_{i-1}}(\xprod) \Qprod(\xprod, \1_{\CLDsetprod_{i}} \fprod) + \1_{\CLDsetprod_{i-1}}(\xprod)\Qprod_i^1(\xprod,\1_{\CLDsetprod_{i}^\comp} \fprod ) + \1_{\CLDsetprod_{i-1}^\comp}(\xprod)\Qprod_i^1(\xprod,\fprod )\eqsp,\\
   &\leq&  \Qprod\left( \xprod, \rho_i^{\1_{\CLDsetprod_{i-1}}(\xprod) \1_{\CLDsetprod_{i}}} \fprod \right)\eqsp.
\end{eqnarray*}
We write $\Delta_n \big( \nu,\nu',\yvect) = \sup_{A \in \Xsigma} |\Delta_n(A)|$, where
\begin{equation*}
\Delta_n(A) \stackrel{{\rm def}}{=} \nu \otimes \nu' \big( \gprod_0 \Qprod \gprod_1 \ldots \Qprod \gprod_{n} \1_{ A \times \Xset} \big) -
  \nu' \otimes \nu \big( \gprod_0 \Qprod \gprod_1 \ldots \Qprod \gprod_{n} \1_{ A \times \Xset} \big) \eqsp.
\end{equation*}
We decompose $\Delta_n(A)$ into $\Delta_n(A) = \sum_{t_{0:n-1} \in \{0,1\}^n} \Delta(A,t_{0:n-1})$, where
\begin{multline*}
\Delta_n(A,t_{0:n-1}) \eqdef \nu \otimes \nu' \big( \gprod_0 \Qprod_0^{t_0} \gprod_1 \ldots \Qprod_{n-1}^{t_{n-1}} \gprod_{n} \1_{ A \times \Xset} \big) \\ -
  \nu' \otimes \nu \big( \gprod_0  \Qprod_0^{t_0} \gprod_1 \ldots \Qprod_{n-1}^{t_{n-1}}\gprod_{n} \1_{ A \times \Xset} \big)\eqsp.
\end{multline*}
Note that, for any $t_{0:n-1} \in \{0,1\}^n$ and any sets $A,B \in \Xsigma$,
\begin{equation*}
    \nu \otimes \nu' \big( \gprod_0 \Qprod_0^{t_0} \gprod_1 \ldots \Qprod_{n-1}^{t_{n-1}} \gprod_{n} \1_{ A \times B} \big) =
    \nu' \otimes \nu \big( \gprod_0 \Qprod_0^{t_0} \gprod_1 \ldots \Qprod_{n-1}^{t_{n-1}} \gprod_{n} \1_{ B \times A } \big)\eqsp.
\end{equation*}
If there is an index $i \in \{0,\dots,n-1\}$ such that $t_i=0$, then
\begin{align*}
    &\nu \otimes \nu' \big( \gprod_0 \Qprod_0^{t_0} \gprod_1 \ldots \Qprod_{n-1}^{t_{n-1}} \gprod_{n} \1_{ A \times \Xset}) \\
    & = \nu \otimes \nu' \big( \gprod_0 \Qprod_0^{t_0} \gprod_1 \ldots \Qprod_{i-1}^{t_{i-1}} \gprod_{i} \1_{ \CLDsetprod_i}) \times  (\varepsilon_{i+1}^-)^2 \lambdaprod_{i} \big( \1_{\CLDsetprod_{i+1}} \gprod_{i+1}\Qprod_{i+1}^{t_{i+1}} \ldots \Qprod_{n-1}^{t_{n-1}}\gprod_{n} \1_{ A \times \Xset} \big) \eqsp, \\
    &= \nu' \otimes \nu \big( \gprod_0 \Qprod_0^{t_0} \gprod_1 \ldots \Qprod_{i-1}^{t_{i-1}} \gprod_{i} \1_{ \CLDsetprod_i}) \times  (\varepsilon_{i+1}^-)^2 \lambdaprod_{i} \big( \1_{\CLDsetprod_{i+1}} \gprod_{i+1}\Qprod_{i+1}^{t_{i+1}} \ldots \Qprod_{n-1}^{t_{n-1}}\gprod_{n} \1_{ A \times \Xset} \big) \eqsp.
\end{align*}
Thus, $\Delta_n(A,t_{0:n-1})=0$ except if for all $i \in \{0,\dots,n-1\}$, $t_i=1$, and we obtain
\begin{equation*}
    \Delta_n(A) = \nu \otimes \nu' \Big[ \gprod_0 \Qprod_0^1 \gprod_1 \ldots \Qprod_{n-1}^1 \gprod_{n} \big(\1_{ A \times \Xset}-\1_{\Xset \times A}   \big)\Big]\eqsp.
\end{equation*}
It then follows
\begin{equation*}
    \Delta_n \big( \nu,\nu',\yvect \big) \leq \nu \otimes \nu'( \gprod_0 \Qprod_0^{1} \gprod_1 \ldots \Qprod_{n-1}^{1} \gprod_{n} )
    \leq \PE_{\nu\otimes \nu'}^{\Qprod}\left[ \gprod(\Xprod_0,y_0) \prod_{i=1}^{n}
    \gprod(\Xprod_{i},y_{i}) \rho_{i}^{\delta_{i}}  \right]\eqsp,
\end{equation*}
with $ \delta_i =  \1_{\CLDsetprod_{i-1} \times \CLDsetprod_{i}} (\Xprod_{i-1},\Xprod_{i})$.
\end{proof}

\begin{proof}[Proof of Proposition \ref{prop:denominator}]
Since $\CLDset$ is an LD-set function, there exist some applications $\varepsilon_{\CLDset}^-,\ \varepsilon_{\CLDset}^+$ such that, for all  $i=1,\ldots,n$, for all $x \in \CLDset(y_{i-1})$ and for all $A \in \Xsigma$ with $A \subset \CLDset(y_{i})$,
\begin{equation}
\label{eq:LDsetlocal-property}
        \varepsilon_{\CLDset}^-(y_{i-1},y_{i}) \lambda_{y_{i-1},y_{i}}(A) \leq Q(x,A) \leq \varepsilon_{\CLDset}^+(y_{i-1},y_{i}) \lambda_{y_{i-1},y_{i}}(A)\eqsp.
\end{equation}
Let us write the obvious inequality
\begin{equation*}
    \PE_{\nu}^{Q} \left[ \prod_{i=0}^{n} g(X_i,y_i)  \right] \geq
    \PE_{\nu}^{Q}\left[ g(X_0,y_0) \prod_{i=1}^{n} g(X_i,y_i) \1_{\CLDset(y_i)}(X_i) \right]\eqsp.
\end{equation*}
Then, for the right-hand side of this expression, by (\ref{eq:LDsetlocal-property}) we have
\begin{align*}
    &\PE_{\nu}^{Q} \left[ g(X_0,y_0) \prod_{i=1}^{n} g(X_i,y_i) \1_{\CLDset(y_i)}(X_i) \right]\\
    & = \PE_{\nu}^{Q}\left[ g(X_0,y_0)g(X_1,y_1) \1_{\CLDset(y_1)}(X_1)
    \prod_{i=2}^{n} g(X_i,y_i) \1_{\CLDset(y_{i-1}) \times \CLDset(y_i)}(X_{i-1},X_i) \right]\eqsp, \\
    & \geq \nu \big[ g(\cdot,y_0)Q g(\cdot,y_1) \1_{\CLDset(y_1)}(\cdot) \big] \prod_{i=2}^{n}\varepsilon_{\CLDset}^-(y_{i-1},y_{i}) \lambda_{y_{i-1},y_i} \big[ g(\cdot,y_i) \1_{\CLDset(y_i)} \big]\eqsp.
\end{align*}
\end{proof}

%% file: PropSimpleExample.tex
\section{Proofs of Propositions \ref{prop:SimpleExamplePathwiseCV} and \ref{prop:SimpleExampleExpectationCV} }
\label{sect:Proof_SimpleModel}

\begin{proof}[Proof of Proposition \ref{prop:SimpleExamplePathwiseCV}]
Since, by definition \eqref{eq:qinf_qsupp_SimpleModel}, $\gamma^-$ is a decreasing function, the inequality \eqref{eq:ControlYproc_SimpleExample} leads to
\begin{equation}
\label{eq:vm_SimpleExample}
     n^{-1} \sum_{k=2}^{n} \log
    \vm_{\Delta}(Y_{k-1},Y_k) \geq -n^{-1}
    \sum_{k=2}^{n} Z_k^{\Delta}\eqsp,
\end{equation}
where $Z_k^\Delta$ is defined in \eqref{eq:definition-ZkDelta}.
Since the process $\left\{ab|\varepsilon_{k-1}|+|\zeta_k|  +b|\varepsilon_k| \right\}_{k\geq 1}$ is stationary 2-dependent, the strong law of large numbers for $m$-dependent sequences and the integrability condition (\ref{hyp:SimpleModel_moment}) yield
\begin{equation}
\label{eq:SLLN_vm_SimpleExample}
    \lim_{n \to \infty} n^{-1} \sum_{k=2}^{n}  Z_k^{\Delta} = \PE (Z_1^\Delta) < \infty \eqsp, \qquad \tPP-\as\
\end{equation}

By combining \eqref{eq:vm_SimpleExample} and \eqref{eq:SLLN_vm_SimpleExample}, the first condition \eqref{hyp:PathwiseConvergence_vm} of Theorem \ref{thm:PathwiseConvergence} is satisfied. By assumption (E\ref{hyp:likelihood}), the density $\upsilon$ is bounded which implies that $\sup_{y \in \Yset} \Upsilon_\Xset(y) \leq  \sup \upsilon$. Hence, the second condition \eqref{hyp:PathwiseConvergence_Upsilon} of Theorem \ref{thm:PathwiseConvergence} is satisfied. We now consider the third condition \eqref{hyp:PathwiseConvergence_Psi}. Since the measure appearing in the definition of the LD-set function does not depend on $y, \ y'$, the function $(y,y') \mapsto \Psi_{\CLDset(y',\Delta)}(y,y')$, defined in \eqref{eq:definition-Psi}, does not depend on $y$ and is given by
\begin{equation*}
    \Psi_{\CLDset(y',\Delta)}(y,y')= \int_{\CLDset(y',\Delta)} \upsilon[y'-h(x)] \, \lleb(dx) \geq \lleb[\CLDset(y',\Delta)] \times \inf_{|s| \leq \Delta} \upsilon(s) \eqsp.
\end{equation*}
Since the function $h$ is uniformly continuous, for any fixed $\Delta > 0$, there exist $\delta>0$ such that, for all $x,\ x' \in \Xset$ satisfying $|x-x'| \leq \delta$, we have $|h(x)-h(x')| \leq \Delta$, showing that $\lleb[C(y',\Delta)] \geq \delta$. Thus, we have, for all $y, \ y' \in \Yset$,
\begin{equation}
\label{eq:Minoration_Psi}
    \Psi_{\CLDset(y',\Delta)}(y,y') \geq \varrho_\Delta \eqsp,
\end{equation}
for some $\varrho_\Delta>0$, depending only on $\Delta$.  The
third condition \eqref{hyp:PathwiseConvergence_Psi} of Theorem \ref{thm:PathwiseConvergence} follows. Since assumption (H\ref{hyp:OutsideLDset_Model}) is satisfied, for any fixed $\eta>0$, we choose $\Delta>0$ such that inequality \eqref{eq:SimpleModel_Upsilon} holds. Let us write
\begin{equation}
\label{eq:R_Delta}
    R_\Delta(x) \eqdef \log\big[1-(\gamma^-/\gamma^+)^2(2c+d\Delta+x) \big]\eqsp.
\end{equation}
We will repeatedly use the following representation of the so-called $L$-statistic (see \cite[Chapter 8]{serfling:1980}):
\begin{lem}
\label{lem:Lstat_CumulFunction}
Let $\{U_1, \ldots,U_n \}$ be a sequence and $U_{n,1}\leq U_{n,2}\leq \ldots \leq U_{n,n}$ the upper ordered statistic. Then,
\begin{equation*}
    n^{-1} \sum_{k=j}^{n} U_{n,k} = \int_{j/n}^{1} F^{-1}_{n,U}(s) \, ds
\end{equation*}
where $F^{-1}_{n,U}(s) \eqdef \inf \{ t  \in \rset, \ F_{n,U}(t) \geq s \}$ is the empirical quantile function, \ie\ the generalized inverse of the empirical distribution function $F_{n,U}(t)  \eqdef  n^{-1} \sum_{k=1}^{n} \1_{\{U_k \leq t \}}$.
\end{lem}
Applying this representation yields
\begin{equation}
\label{eq:Lambda_CumulFunc_SimpleExample}
    n^{-1} \log \Lambda_\eta( Y_{0:n},\alpha)
     \leq \int_{0}^{1} \1\{u \geq 1-r_n\} F_n^{-1}(u)\,
   du \eqsp,
\end{equation}
where $r_n =(\lceil n\alpha\rceil-1)/n$, $F_n(t)= n^{-1} \sum_{k=1}^{n} \1\{ R_\Delta(ab|\varepsilon_{k-1}|+|\zeta_k|  +b|\varepsilon_k|) \leq t\}$ and $F_n^{-1}$ its generalized inverse. The function $R_\Delta$  defined by \eqref{eq:R_Delta} is negative and then, $F_n(0)=1$ which implies that $F_n^{-1}(u) \geq 0$ for all $u\in (0,1)$. Thus, by Fatou's lemma,
\begin{multline}\label{eq:Fatou_SimpleExample}
    \limsup_{ n \to \infty} \int_{0}^{1} \1\{u \geq 1-r_n\} F_n^{-1}(u)\, du \\
    \leq \int_{0}^{1}  \limsup_{ n \to \infty} \1\{u \geq 1-r_n\} F_n^{-1}(u)\, du \quad \tPP-\as\
\end{multline}
The following lemma is a generalization of \cite[Lemma 21.2]{vandervaart:1998}.
\begin{lem}
\label{lem:InverseFunc_CV}
Let $\{ \Psi_n\}_{n \geq 0}$ be a sequence of nondecreasing functions and $\Psi$ a bounded nondecreasing function such that for all $x \in \Xset$,  $\lim_{n \to \infty} \Psi_n(x)= \Psi(x)$. Then, $\Psi^{-1}$ has at most a countable number of discontinuity points
and at any point $u$ where $\Psi^{-1}$ is continuous,
\begin{equation*}
    \lim_{n \to \infty} \Psi_n^{-1}(u)  = \Psi^{-1}(u) \eqsp.
\end{equation*}
\end{lem}
Let us denote $F(t) = \PP [R_\Delta(ab|\varepsilon_0|+|\zeta_1|  +b|\varepsilon_1|) \leq t ]$ and notice that $F(0)=  1$. Then, combining \eqref{eq:Lambda_CumulFunc_SimpleExample}, \eqref{eq:Fatou_SimpleExample} and Lemma \ref{lem:InverseFunc_CV} leads to
\begin{equation*}
    \limsup_{n \to \infty} n^{-1} \log \Lambda_\eta( Y_{0:n},\alpha) \leq \int_{1-\alpha}^{1} F^{-1}(u)\, du <0 \eqsp, \quad \tPP-\as\ \eqsp.
\end{equation*}
This shows that the fourth condition  \eqref{hyp:PathwiseConvergence_Lambda} is satisfied and finally, Theorem \ref{thm:PathwiseConvergence} applies.

\end{proof}

Let us recall that $\psi_{Z}$ denotes the log-moment generating function of the random variable $Z$ defined by $\psi_Z(\lambda) \eqdef \log \PE[ \rme^{\lambda Z}]$ and we define its Legendre's transformation by
\begin{equation*}
    \psi_Z^*(x)= \sup_{\lambda \geq 0} \left\{x \lambda - \psi_Z(\lambda) \right\} \eqsp.
\end{equation*}

\begin{proof}[Proof of Proposition \ref{prop:SimpleExampleExpectationCV}]
We start by giving an exponential inequality for $m$-dependent variables.
\begin{lem}
\label{lem:ExpInequality_m-dependent}
Let $\Zproc$ be a sequence of $m$-dependent stationary random variables. There exists some constant $C>0$ such that, for all  $M \geq 0$,
\begin{equation*}
    \PP\left( \sum_{k=1}^{n} Z_k \geq  M n \right) \leq C  \exp[ -n  \psi_{Z_1}^*(2 M m )/(2m)] \eqsp.
\end{equation*}
\end{lem}
The proof is elementary and left to the reader. It follows by equation \eqref{eq:vm_SimpleExample} that
\begin{equation*}
    \PP\left( n^{-1} \sum_{k=2}^{n} \log
    \vm_{\Delta}(Y_{k-1},Y_k) \leq -M_1 n \right) \leq \PP\left(  \sum_{k=2}^{n} Z_k^\Delta  \geq M_1 n \right) \eqsp.
\end{equation*}
Thanks to \eqref{hyp:SimpleModel_psi}, by applying Lemma \ref{lem:ExpInequality_m-dependent}, there exist some constant $c_1, \, \delta_1 >0$ such that $r_1(n) \leq c_1 \rme^{-\delta_1 n}$. Since $\upsilon$ is bounded, we can choose $M_2$ large enough such that $r_2(n)=0$. By \eqref{eq:Minoration_Psi}, for all $(y,y') \in \Yset^2$, $\Psi_{\CLDset(y',\Delta)}(y,y') \geq \varrho_\Delta \eqsp$, for some $\varrho_\Delta>0$. Then, by choosing $M_3$ large enough, we have $r_3(n)=0$.
For $r_4(n)$, we need an exponential inequality for $L$-statistics based on $m$-dependent variables.
\begin{lem}
\label{lem:ExpInequality_L-stat_m-dep}
Let $\Uproc$ be a sequence of  $m$-dependent stationary negative random variables. For all $\alpha \in (0,1)$, there exists a  real $r>0$  such that
\begin{equation*}
    \lim_{n \to \infty} n^{-1} \log \PP\left(\sum_{k=n-\lceil \alpha n \rceil +1}^{n} U_{n,k} \geq -r n \right) <0 \eqsp.
\end{equation*}
\end{lem}
\begin{proof}[Proof of Lemma \ref{lem:ExpInequality_L-stat_m-dep}] For $j \in \{1,\dots,m\}$, define $I_j=\{j, j+m, j+2m, \ldots \}$ and let $n_j = |I_j|$ the cardinal of $I_j$. For any $j \in \{1,\dots,m\}$, the sequence $\{ U_k, k \in I_j\}$ is  i.i.d..
Denote $\{U^{(j)}_k \}_{1 \leq k\leq n_j}$ the sequence $\{ U_k, k \in I_j \}$.
Since $U_k <0$ for all integer $k$, it then follows that
\begin{equation*}
    \sum_{k=n-\lceil \alpha n \rceil +1}^{n} U_{n,k} \leq \sum_{j=1}^m  \sum_{k=(n_j -\lceil \alpha n \rceil +1) \vee 0 }^{n_j} U^{(j)}_{n_j,k} \eqsp,
\end{equation*}
and
\begin{equation*}
    \PP\left(\sum_{k=n-\lceil \alpha n \rceil +1}^{n} U_{n,k} \geq -r n \right) \leq
    \sum_{j=1}^{m} \PP\left(   \sum_{k=(n_j - \lceil \alpha n \rceil  +1) \vee 0}^{n_j} U^{(j)}_{n_j,k} \geq -r n / m \right)\eqsp,
\end{equation*}
for all $n\geq N$ larger than some integer $N$. The sequence $\{U^{(j)}_k\}_{1\leq k \leq n_j}$ is a sequence of i.i.d. random variable. Then, using \cite[Theorem 6.1]{groeneboom:oosterhoff:ruymgaart:1979}, we have
\begin{equation*}
    \lim_{n \to \infty} n_j^{-1} \log \PP\left(  n_j^{-1} \sum_{k=(n_j -\lceil \alpha n \rceil  +1) \vee 0}^{n_j} U^{(j)}_{n_j,k} \geq -\delta \right) <0 \eqsp,
\end{equation*}
for some positive $\delta$ and the result follows since $n_j/n= 1/m + o(1)$.

\end{proof}
Define by $U_k = R_\Delta \left[ ab|\varepsilon_{k-1}|+|\zeta_k| +b|\varepsilon_k| \right]$ for all integer $k \geq 1$. By the definition \eqref{eq:defn_Lambda} of $\Lambda_\eta$,
\begin{equation}
    \label{eq:L-stat}
     n^{-1} \log \Lambda_\eta( Y_{0:n},\alpha)
     \leq n^{-1} \sum_{k=n-\lceil \alpha n \rceil +1}^{n} U_{n,k} \eqsp.
\end{equation}
Then, by equation \eqref{eq:L-stat} and by applying Lemma \ref{lem:ExpInequality_L-stat_m-dep}, there exist some constants $c_4,\, \delta_4 >0$ such that $r_4(n) \leq c_4 \rme^{-\delta_4 n}$. Finally, under assumptions (E1), (E2) and \eqref{hyp:SimpleModel_psi}, Theorem \ref{thm:ExpectationConvergence} applies and provides a geometric rate.

\end{proof}

%

%% file: PropGenericExample.tex
\section{Proofs of Propositions \ref{prop:GenericExamplePathwiseCV} and \ref{prop:GenericExampleExpectationCV} }
\label{sect:Proofs_Generic}

\begin{proof}[Proof of Proposition \ref{prop:GenericExamplePathwiseCV}]
Let us define, for all $\Delta>0$ and for all integer $k\geq 1$,
\begin{equation}
\label{eq:def_Vkstar}
     V^{*\Delta}_k = \log q^-\left[c+d\Delta +\kappa +2a^*b^* + a^*b^* |\eps_{k-1}^*| +b^*|\eps_k^*|+|\zeta_k^*| \right]\eqsp.
\end{equation}
Using the definitions \eqref{eq:definition_q_pm}, \eqref{eq:definition_varepsilon_pm} of $q^-$ and  $\vm_\Delta$, Lemma \ref{lem:D_Ystar} shows that
\begin{equation}
\label{eq:GenProof_pathwise_1}
    n^{-1} \sum_{k=2}^{n} \vm_\Delta(Y_{k-1}^*,Y_k^*) \geq n^{-1} \sum_{k=2}^{n} V^{*\Delta}_k \eqsp.
\end{equation}
Thus, to check \eqref{hyp:PathwiseConvergence_vm}, it suffices to control the asymptotic behavior of  right-hand side of this inequality.
We  use the following  result \cite[Chapter 2, Section 6]{hall:heyde:1981}.
\begin{lem}
\label{lem:Martingal_convergence}
Let us denote by \Hsigmaproc a filtration and consider a sequence
$\{ U_k \}_{k \geq 0}$ of random variable adapted to \Hsigmaproc. Let us assume that there exists a random variable $U$ such that $\PE \left( |U| \log_+|U| \right) < \infty$ and $\PP(|U_k|>x) \leq c\, \PP(|U| > x)$ for all $x>0$ and some $c>0$. Then
\begin{equation*}
    \lim_{n \to \infty} n^{-1} \sum_{k=1}^{n} \left[ U_k - \PE(U_k | \Hsigma_{k-1}) \right] = 0 \eqsp, \qquad \PP-\as\
\end{equation*}
\end{lem}

Define the filtration $\{\Fsigma_k^*\}_{k \geq 0}$ where $
    \Fsigma_k^*= \sigma\left( \{X^*_j\}_{0\leq j \leq k}, \{\zeta^*_j\}_{0\leq j \leq k},\{\eps^*_j\}_{j \geq 0} \right)$. Since $q_-$ defined in \eqref{eq:definition_q_pm} is non-increasing, there exists $c > 0$ such that, for all $x > 0$, $\PP(|V^{*\Delta}_k| > x) \leq c \PP(|V_+^{*\Delta}| > x)$, where $V_+^{*\Delta}$ is defined in \eqref{eq:definition-V*+}. Hence, we may apply Lemma \ref{lem:Martingal_convergence} which yields, for any $\Delta>0$,
\begin{equation}
\label{eq:GenProof_pathwise_2}
    \liminf_{n \to \infty}
    n^{-1} \sum_{k=2}^{n}   V^{*\Delta}_k= \liminf_{n \to \infty}
    n^{-1} \sum_{k=2}^{n} \PE\{V^{*\Delta}_k | \Fsigma_{k-1}^*\} \eqsp, \qquad \tPP-\as\
\end{equation}
By \eqref{eq:GenExObs_Forgetting}, since for all $x>0$, $\log x \geq -\log_-x$, then, by the strong law of large numbers,
\begin{equation}
\label{eq:GenProof_pathwise_3}
    \liminf_{n \to \infty} n^{-1} \sum_{k=2}^{n} \PE\{ V^{*\Delta}_k | \Fsigma_{k-1}^*\}  \geq  -\PE [H_\Delta(a^*b^* |\eps_{0}^*| +b^*|\eps_1^*|)] \eqsp, \qquad \tPP-\as\
\end{equation}
where $H_\Delta(x)= \mu^*_+ \times  \int \log_- q^{-}[c+d\Delta +\kappa +2a^*b^* + x+|w| ] \psi^*(w) \, dw$. By \eqref{eq:Martingale_CV},
$ \PE [H_\Delta(a^*b^* |\eps_{0}^*| +b^*|\eps_1^*|)] < \infty$, it then follows by \eqref{eq:GenProof_pathwise_1}, \eqref{eq:GenProof_pathwise_2} and \eqref{eq:GenProof_pathwise_3} that
\begin{equation*}
    \liminf_{n \to \infty} n^{-1} \sum_{k=2}^{n} \log
    \vm_{\Delta}(Y_{k-1}^*,Y_k^*) \geq \liminf_{n \to \infty}
    n^{-1} \sum_{k=2}^{n}  V^{*\Delta}_k  > -\infty \eqsp,\qquad \tPP-\as\
\end{equation*}
and the condition \eqref{hyp:PathwiseConvergence_vm} is satisfied.
The proof of assumptions \eqref{hyp:PathwiseConvergence_Upsilon} and \eqref{hyp:PathwiseConvergence_Psi} can be checked as in  Proposition \ref{prop:SimpleExamplePathwiseCV}.
Since (H\ref{hyp:OutsideLDset_Model}) is satisfied, for a fixed $\eta>0$, we choose $\Delta>0$ such that $\Upsilon_{\CLDset^\comp(y,\Delta)}(y)\leq \eta \Upsilon_\Xset(y)$. Applying Lemma \ref{lem:Lstat_CumulFunction} yields
\begin{equation*}
    n^{-1} \log \Lambda_\eta( Y_{0:n}^*,\alpha)
     \leq \int_{0}^{1} \1\{1-r_n \leq u\} {F_n^*}^{-1}(u)\,
   du \eqsp,\qquad \tPP-\as\
\end{equation*}
where $r_n =(\lceil n\alpha\rceil-1)/n$ and ${F_n^*}^{-1}$ is the generalized inverse of the distribution function:
\begin{equation}\label{eq:F_n}
    F_n^*(t)= n^{-1} \sum_{k=1}^{n} \1\{R_\Delta[\kappa +2a^*b^* + a^*b^* |\eps_{k-1}^*| +b^*|\eps_k^*|+|\zeta_k^*|] \leq t\} \eqsp,
\end{equation}
with $R_\Delta$ is defined in \eqref{eq:R_Delta}. For convenience, let us write $G(\eps_{k-1}^*,\eps_k^*,\zeta_k^*)= R_\Delta[\kappa +2a^*b^* + a^*b^* |\eps_{k-1}^*| +b^*|\eps_k^*|+|\zeta_k^*|]$.
Setting
\begin{equation}
\label{eq:H_n}
    H_n^*(t) = n^{-1} \sum_{k=1}^{n}  \PP\left\{  G(\eps_{k-1}^*,\eps_k^*,\zeta_k^*) \leq t |  \Fsigma_{k-1}^* \right\} \eqsp,
\end{equation}
it  follows from Lemma \ref{lem:Martingal_convergence} that, for a fixed $t \in \R$,
\begin{equation}
\label{eq:DiffCumulFunction_Convergence}
    \lim_{n \to \infty} \{ F_n^*(t)-H_n^*(t)\} = 0\eqsp, \qquad \tPP-\as\
\end{equation}
The convergence in \eqref{eq:DiffCumulFunction_Convergence} may be shown to hold uniformly in $t$:
\begin{lem}
\label{lem:DiffCumulFunction_UnifConvergence} Let us consider the stochastic functions $F_n^*$ and $H_n^*$ defined by \eqref{eq:F_n}, \eqref{eq:H_n}. Then,
\begin{equation}
\label{eq:GenProof_CVunif}
    \lim_{n \to \infty} \norminfty{F_n^*-H_n^*} =0 \eqsp, \qquad \tPP-\as\
\end{equation}
\end{lem}
\begin{proof}
Let us define
\begin{align}
\label{eq:definition_J}
    &J_n^*(t)= n^{-1} \sum_{k=1}^{n} \int \PP \left. \left\{ G(\eps_{k-1}^*,\eps_k^*,w) \leq t \right.| \Fsigma_{k-1}^* \right\} \psi^*(w)\,dw
    \eqsp, \\
    &J^*(t)= \PE \left[ \int \1\{G(\eps_0^*,\eps_1^*,w)\leq t\} \psi^*(w)\,dw \right].
\end{align}
By the Glivenko-Cantelli Theorem, $\lim_{n \to \infty} \| J_n^* -J^*\|_{\infty} =0$, $\tPP$-\as\
Set $\varepsilon >0$ and a sequence $-\infty=t_0 \leq t_1 \ldots \leq t_N=\infty$ such that $J^*(t_i^-)-J^*(t_{i-1})<\varepsilon/\mu^*_+$ for every $i$.
By \eqref{eq:GenExObs_Forgetting}, for all real numbers  $t<t'$, $\tPP$-\as\
\begin{equation*}
    H_n^*(t')-H_n^*(t) =   n^{-1} \sum_{k=1}^{n}   \PP(  t<  G(\eps_{k-1}^*,\eps_k^*,\zeta_k^*) \leq t' |  \Fsigma_{k-1}^*)\\
    \leq  \mu^*_+ [J_n^*(t')-J_n^*(t)]\eqsp,
\end{equation*}
and then
\begin{equation*}
    \limsup_{n \to \infty} |H_n^*(t')-H_n^*(t)| \leq \mu^*_+ |J^*(t')-J^*(t)|\eqsp,\qquad \tPP-\as\
\end{equation*}
For all $t \in \rset$, there exists  an index $i$ such that $t_{i-1} \leq t < t_{i}$. Since $F_n^*$ and $H_n^*$ are increasing functions, it follows that
\begin{equation*}
F_n^*(t_{i-1}) \leq F_n^*(t) \leq F_n^*(t_i^-)\eqsp, \quad H_n^*(t_{i-1}) \leq H_n^*(t) \leq H_n^*(t_{i}^-)\eqsp.
\end{equation*}
These inequalities imply
\begin{equation*}
    \sup_{t \in \rset} |F_n^*(t)-H_n^*(t)| \leq \max_{0 \leq i \leq N}|F_n^*(t_i^-)-H_n^*(t_i^-)| + \max_{1 \leq i \leq N}|H_n^*(t_i^-)-H_n^*(t_{i-1})| \eqsp,
\end{equation*}
and then
\begin{equation*}
    \limsup_{n \to \infty} \sup_{t \in \rset} |F_n^*(t)-H_n^*(t)| \leq  \varepsilon \eqsp, \qquad \tPP-\as\
\end{equation*}
\end{proof}

By \eqref{eq:GenExObs_Forgetting}, for all $t \in \rset$,
\begin{equation*}
  F_n^*(t) = F_n^*(t) - H_n^*(t) +H_n^*(t) \geq   F_n^*(t) - H_n^*(t)  + \mu^*_- J_n^*(t)\eqsp, \qquad \tPP-\as\
\end{equation*}
Hence, using the limit \eqref{eq:GenProof_CVunif}, for a given $\delta>0$, there exists an integer $l$ such that, for all $n
\geq l$ and  $t \in \rset$,
\begin{equation}
\label{eq:CumulFunc_GenericExample}
     F_n^*(t) \geq \mu^*_- J_n^*(t)- \delta\eqsp, \qquad \tPP-\as\
\end{equation}
Let us notice that $J_n^*$ is an increasing function with $\lim_{t \to -\infty} J_n^*(t) =0$ and $\lim_{t \to +\infty} J_n^*(t) = 1$. Then, we can define its generalized inverse denoted by ${J_n^*}^{-1}$. By \eqref{eq:CumulFunc_GenericExample}, it follows that, for all $u \in [0,(\mu^*_- - \delta)\wedge 0]$,
\begin{equation*}
    {F_n^*}^{-1}(u) \leq  {J_n^*}^{-1} [(u+\delta)/\mu^*_-]\eqsp, \qquad \tPP-\as\
\end{equation*}
By choosing $\delta>0$ such that $\mu^*_- -\delta>1-\alpha$, there exists an integer $i\geq l$ such that, for all $n\geq i$, we have
\begin{multline*}
    \int_{0}^{1} \1\{ 1-r_n \leq u\} {F_n^*}^{-1}(u)\, du  \\
    \leq \int_{0}^{1} \1\{ 1-r_n \leq u \leq \mu^*_- -\delta\} {J_n^*}^{-1}[(u+\delta)/\mu^*_-] \, du \eqsp,\qquad \tPP-\as\
\end{multline*}
By Fatou's lemma,
\begin{multline*}
    \limsup_{n \to \infty} \int_{0}^{1} \1\{ 1-r_n \leq u\} {F_n^*}^{-1}(u)\, du  \\
    \leq
    \int_{0}^{1} \limsup_{n \to \infty} \1\{ 1-r_n \leq u\leq \mu^*_- -\delta\} {J_n^*}^{-1}[(u+\delta)/\mu^*_-] \, du \eqsp,\quad \tPP-\as\
\end{multline*}
It follows by Lemma \ref{lem:InverseFunc_CV} that
\begin{equation*}
    \limsup_{n \to \infty}   n^{-1} \log \Lambda_\eta( Y_{0:n}^*,\alpha) \leq \int_{1-\alpha}^{\mu^*_- -\delta} {J^*}^{-1}[(u+\delta)/\mu^*_-] \, du  < 0 \eqsp, \qquad \tPP-\as\
\end{equation*}
Thus, condition \eqref{hyp:PathwiseConvergence_Lambda} is satisfied and Theorem \ref{thm:PathwiseConvergence} applies.

\end{proof}

\begin{proof}[Proof of Proposition \ref{prop:GenericExampleExpectationCV}]
It follows, by definition of $r_1$, Lemma \ref{lem:D_Ystar} and (O3), that
\begin{multline*}
    r_1(n) = \tPP\left( n^{-1} \sum_{k=2}^{n} \log q^{-}[c+d\Delta +D(Y_{k-1}^*,Y_k^*)] \leq -M_1 n \right) \eqsp  \leq \\
    \tPP\left(  n^{-1} \sum_{k=2}^{n} \log q^-\left[c_0 + a^*b^* |\eps_{k-1}^*| +b^*|\eps_k^*|+g^*_+(A_k^*) \right] \leq -M_1 n \right) \eqsp.
\end{multline*}
with $c_0= c+d\Delta +\kappa +2a^*b^*$. Then, by (O4) and applying Lemma \ref{lem:ExpInequality_m-dependent}, there exist some constants $c_1, \, \delta_1 >0$ such that $r_1(n) \leq c_1 \rme^{-\delta_1 n}$. By the same arguments as in proof of Proposition \ref{prop:SimpleExampleExpectationCV}, the real numbers $M_2$ and $M_3$ can be chosen large enough such that $r_2(n)=0$ and $r_3(n)=0$. Let us denote by $\{U^+_k\}_{k \geq 0}$ the sequence defined by  $U^+_k =R_\Delta[\kappa +2a^*b^* + a^*b^* |\eps_{k-1}^*| +b^*|\eps_k^*|+g^*_+(A_k^*)]$, for all integer $k\geq 1$. By definition of $\Lambda_\eta$,
\begin{equation}
     n^{-1} \log \Lambda_\eta( Y_{0:n}^*,\alpha)
     \leq n^{-1} \sum_{k=n-\lceil \alpha n \rceil +1}^{n} U^+_{n,k} \eqsp.
\end{equation}
By applying Lemma \ref{lem:ExpInequality_L-stat_m-dep}, there exist some constants $c_4,\, \delta_4 >0$ such that $r_4(n) \leq c_4 \rme^{-\delta_4 n}$. Finally, Theorem \ref{thm:ExpectationConvergence} applies and provides a geometric rate.

\end{proof}